\documentclass{article}
\usepackage[margin=1in]{geometry} 
\usepackage{amssymb,amsfonts,amsmath,bbm,mathrsfs,stmaryrd}
\usepackage{xcolor}
\usepackage{url}
\usepackage{dsfont}
\usepackage{enumerate}
\usepackage{enumitem}

\usepackage[colorlinks,
             linkcolor=black!75!red,
             citecolor=blue,
             pdftitle={},
             pdfauthor={},
             pdfproducer={pdfLaTeX},
             pdfpagemode=None,
             bookmarksopen=true
             bookmarksnumbered=true]{hyperref}

\usepackage{tikz}
\usetikzlibrary{arrows,calc,decorations.pathreplacing,decorations.markings,intersections,shapes.geometric,through,fit,shapes.symbols,positioning,decorations.pathmorphing}

\usepackage{braket}

\usepackage[amsmath,thmmarks,hyperref]{ntheorem}
\usepackage{cleveref}

\creflabelformat{enumi}{#2(#1)#3}

\crefname{section}{Section}{Sections}
\crefformat{section}{#2Section~#1#3} 
\Crefformat{section}{#2Section~#1#3} 

\crefname{subsection}{\S}{\S\S}
\crefformat{subsection}{#2\S#1#3} 
\Crefformat{subsection}{#2\S#1#3} 

\theoremstyle{plain}

\newtheorem{lemma}{Lemma}[section]
\newtheorem{proposition}[lemma]{Proposition}
\newtheorem{corollary}[lemma]{Corollary}
\newtheorem{theorem}[lemma]{Theorem}

\theoremstyle{nonumberplain}

\theoremstyle{plain}
\theorembodyfont{\upshape}
\theoremsymbol{\ensuremath{\blacklozenge}}

\newtheorem{definition}[lemma]{Definition}
\newtheorem{example}[lemma]{Example}
\newtheorem{remark}[lemma]{Remark}

\crefname{definition}{definition}{definitions}
\crefformat{definition}{#2definition~#1#3} 
\Crefformat{definition}{#2Definition~#1#3} 

\crefname{ex}{example}{examples}
\crefformat{example}{#2example~#1#3} 
\Crefformat{example}{#2Example~#1#3} 

\crefname{remark}{remark}{remarks}
\crefformat{remark}{#2remark~#1#3} 
\Crefformat{remark}{#2Remark~#1#3} 

\crefname{convention}{convention}{conventions}
\crefformat{convention}{#2convention~#1#3} 
\Crefformat{convention}{#2Convention~#1#3}

\crefname{lemma}{lemma}{lemmas}
\crefformat{lemma}{#2lemma~#1#3} 
\Crefformat{lemma}{#2Lemma~#1#3} 

\crefname{proposition}{proposition}{propositions}
\crefformat{proposition}{#2proposition~#1#3} 
\Crefformat{proposition}{#2Proposition~#1#3} 

\crefname{corollary}{corollary}{corollaries}
\crefformat{corollary}{#2corollary~#1#3} 
\Crefformat{corollary}{#2Corollary~#1#3} 

\crefname{theorem}{theorem}{theorems}
\crefformat{theorem}{#2theorem~#1#3} 
\Crefformat{theorem}{#2Theorem~#1#3} 

\crefname{assumption}{assumption}{Assumptions}
\crefformat{assumption}{#2assumption~#1#3} 
\Crefformat{assumption}{#2Assumption~#1#3} 

\crefname{equation}{}{}
\crefformat{equation}{(#2#1#3)} 
\Crefformat{equation}{(#2#1#3)}

\theoremstyle{nonumberplain}
\theoremsymbol{\ensuremath{\blacksquare}}

\newtheorem{proof}{Proof}

\newcommand\bC{{\mathbb C}}

\newcommand\bF{{\mathbb F}}

\newcommand\bN{{\mathbb N}}

\newcommand\bT{{\mathbb T}}

\newcommand\cC{{\mathcal C}}

\newcommand\cG{{\mathcal G}}
\newcommand\cH{{\mathcal H}}

\newcommand\cK{{\mathcal K}}
\newcommand\cL{{\mathcal L}}

\newcommand\cO{{\mathcal O}}

\newcommand\cT{{\mathcal T}}

\newcommand{\Tr}{\text{Tr}}

\newcommand{\op}{{\text{op}}}

\newcommand{\id}{\text{id}}

\newcommand{\<}{\left\langle}
\renewcommand{\>}{\right\rangle}

\title{\textbf{Quantum Edge Correspondences and Quantum Cuntz--Krieger Algebras}}
\author{    Michael Brannan 
            \footnote{Department of Pure Mathematics and the Institute for Quantum Computing, University of Waterloo \hfill \url{michael.brannan@uwaterloo.ca}}
            
            \and
            
            Mitch Hamidi
            \footnote{Department of Mathematics, Embry-Riddle Aeronautical University \hfill \url{hamidim@erau.edu}}
            
            \and
            
            Lara Ismert
            \footnote{Department of Mathematics, Embry-Riddle Aeronautical University \hfill \url{ismertl@erau.edu}}

            \and
            
            Brent Nelson
            \footnote{Department of Mathematics, Michigan State University \hfill \url{brent@math.msu.edu}}
            
             \and
            
            Mateusz Wasilewski
            \footnote{Institute of Mathematics, Polish Academy of Sciences \hfill
            \url{mwasilewski@impan.pl}}
        }

\begin{document}

\date{}

\maketitle

\begin{abstract}
Given a quantum graph $\cG=(B,\psi,A)$, we define a C*-correspondence $E_\cG$ over the noncommutative vertex C*-algebra $B$, called the \emph{quantum edge correspondence}. For  a classical graph $\cG$, $E_\cG$ is the usual graph correspondence spanned by the edges of $\cG$. When the quantum adjacency matrix $A\colon B\to B$ is completely positive, we show that $E_\cG$ is faithful if and only if $\ker(A)$ does not contain a central summand of $B$. In this case, we show that the Cuntz--Pimsner algebra $\cO_{E_\cG}$ is isomorphic to a quotient of the quantum Cuntz--Krieger algebra $\cO(\cG)$ defined in \cite{BEVW20}. Moreover, the kernel of the quotient map is shown to be generated by ``localized'' versions of the quantum Cuntz--Krieger relations, and $\cO_{E_\cG}$ is shown to be the universal object associated to these local relations.  We study in detail some concrete examples and make connections with the theory of Exel crossed products.
\end{abstract}


\section*{Introduction}
The notion of a quantum graph goes back to the work of Erdos-Katavolos-Shulman \cite{ErKaSh98} and Weaver \cite{We12}, and was subsequently developed further by Duan-Severini-Winter \cite{DuSeWi13} and Musto-Reutter-Verdon \cite{MuReVe19}. Quantum graphs play an intriguing role in the study of the graph isomorphism game in quantum information via their connections with quantum symmetries of graphs (see \cite{MR4097284} and \cite{BCEHPSWCMP19}). Moreover, based on the use of quantum symmetries, fascinating results on the graph theoretic interpretation of quantum isomorphisms between finite graphs were recently obtained by Man\v{c}inska-Roberson \cite{MaRo19}.
In this paper, we take a finite directed quantum graph to mean a triple $(B,\psi, A)$ consisting of a finite-dimensional C*-algebra $B$, a state $\psi$ on $B$, and a linear map $A$ on $B$ satisfying a quantum Schur idempotent condition. Finite directed quantum graphs generalize classical finite directed graphs (without multiple edges) by encoding a classical graph $\cG$ in the triple $(B,\psi, A)$, where $B = C(V)$ is the C*-algebra of continuous functions on the vertex set $V$ of $\mathcal G$, $\psi$ is integration with respect to the uniform probability on $V$, and $A$ is the adjacency matrix of $\cG$.   

In \cite{BEVW20}, given a finite directed quantum graph $(B,\psi,A)$ Eifler, Voigt, Weber and first author introduced a novel C*-algebra $\bF\cO(B, \psi, A)$, called the \textit{free quantum Cuntz--Krieger algebra}. This generalizes the well-studied Cuntz--Krieger algebra $\cO_A$ arising from classical graphs 
(or rather a free version of it), where the standard generators are replaced by matrix-valued valued partial isometries whose matrix sizes are determined by the quantum graph, and the Cuntz--Krieger relations are expressed using the quantum adjacency matrix of the quantum graph in analogy to the scalar case.
Introduced in \cite{ck1}, Cuntz--Krieger algebras have intimate connections with symbolic dynamics, and have been studied intensively in the framework of graph C*-algebras over the past decades, thus providing a rich supply of interesting examples \cite{Raeburn, T17, ERRS18, ERRS21}. The structure of graph C*-algebras is understood to an impressive level of detail, and many algebraic properties can be interpreted in terms of the combinatorial properties of the underlying graphs. Motivated by this success, the original constructions and results have been generalized in several directions, including higher rank graphs \cite{KuPa00}, Exel-Laca algebras \cite{ExLa99}, and ultragraph algebras \cite{To03}, among others.
Recall that, under mild assumptions, the  Cuntz--Krieger algebra corresponding to a classical graph is isomorphic to the Cuntz--Pimsner algebra associated to the graph's edge correspondence \cite[Example 2, p. 193]{Pimsner}. It is worth mentioning that in the more general setting of quantum graphs, the free quantum Cuntz--Krieger algebras seem to be difficult to compute in general, and their isomorphism classes are only known under very strict assumptions (e.g., when $(B,\psi,A)$ is \emph{complete} and $\psi$ is an $n^2$-form for $n\in \mathbb{N}$; see \cite[Theorem 4.5]{BEVW20}).

In the present paper we consider a natural unital version of $\mathbb{F}\mathcal{O}(B,\psi,A)$, which we denote $\mathcal{O}(B,\psi,A)$, and under the assumption that $A$ is completely positive we show that $\mathcal{O}(B,\psi,A)$ quotients onto the Cuntz--Pimsner algebra associated to a C*-correspondence over $B$ which can be viewed as the quantum analogue of the edge correspondence for a classical graph. This is accomplished by showing that  this Cuntz--Pimsner algebra is the universal C*-algebra associated to ``local'' versions of the quantum Cuntz--Krieger relations introduced in \cite{BEVW20}.

In Section~\ref{sec:quantumedgecorrespondences}, we define the quantum edge correspondence $E_\cG$ for a given quantum graph $\cG = (B,\psi,A)$ whose associated state is a $\delta$-form. This C*-correspondence is generated by a generalized version of the Choi--Jamio\l kowski matrix $\epsilon_\cG$ associated to the quantum adjacency operator $A$, and $E_\cG$ generalizes the usual edge correspondence for a classical graph (see Example~\ref{exmp:commutativecase}). Given the role of the element $\epsilon_\cG$ in generating this analogue of an edge correspondence, we think of it as the quantum analogue of a classical graph's edge matrix. Theorem~\ref{thm:faithful_full_correspondence} identifies conditions on the quantum adjacency matrix which result in faithfulness and fullness of the quantum edge correspondence, and Proposition \ref{prop:recognition} provides a recognition theorem for when a cyclic C*-correspondence is the quantum edge correspondence for a quantum graph. 
  
In Section~\ref{sec:QCK_and_local_relations}, we introduce a natural quotient of the quantum Cuntz--Krieger algebra $\cO(\cG)$ by introducing certain``local" relations  on the generators, and we call this quotient a {\em local quantum Cuntz--Krieger algebra}. Given a quantum graph $\cG$, we show in Theorem \ref{thm:universal_local_quantum_Cuntz--Krieger_algebra} that the Cuntz--Pimsner algebra $\cO_{E_\cG}$ constructed from the quantum edge correspondence is precisely the local quantum Cuntz--Krieger algebra for that same quantum graph. As mentioned previously, the free quantum Cuntz--Krieger algebras (and their unital cousins) seem to be very difficult to describe concretely, except in the most basic cases.  This sentiment is emphasized by the fact that we are unable to find an example of a quantum graph $\cG$ whose local quantum Cuntz--Krieger algebra is a {\it proper} quotient of $\cO(\cG)$.

In Section~\ref{sec:simpler_examples}, we focus on examples, beginning with a very special class of quantum graphs called the {\it complete quantum graphs} $\cG = K(B,\psi)$.  Here we study the associated quantum Cuntz--Krieger algebras and Cuntz--Pimsner algebras, and make some connections to Exel's theory of crossed products of C*-algebras by endomorphisms \cite{Exel03}. 
In the classical setting, any square $\{0,1\}$-matrix $A = [A(x,y)]$ gives rise to the Markov subshift  $(X_A, \sigma)$, which is the topological dynamical system given by the infinite compact path space $X_A = \{x = (x_i) \in \{1, \ldots, n\}^{\mathbb N}\ \big| \ A(x_i,x_{i+1}) = 1 \ \forall i \ge 1  \}$ together with the left shift action $\sigma:X_A \to X_A$ given by $\sigma(x)_i = x_{i+1}$.  From this dynamical system $(X_A, \sigma)$, one can associate an {\it Exel system} $(C(X_A),\alpha,\mathcal{L})$,  where $C(X_A)$ is the unital C*-algebra of continuous functions on $X_A$,  $\alpha: C(X_A) \to C(X_A)$ is the $*$-endomorphism defined by $\alpha(f) = f \circ \sigma$, and $\mathcal{L}:C(X_A) \to C(X_A)$ is a {\it transfer operator} for $\alpha$. In \cite{Exel03}, Exel builds from this data a crossed product C$^\ast$-algebra $C(X_A)\rtimes_{\alpha, \mathcal{L}} \mathbb{N}$, and shows that it is isomorphic to the usual Cuntz--Krieger algebra $\mathcal{O}_A$.  For any complete quantum graph $K = K(B,\psi)$, we associate a natural choice of (non-commutative) Exel system, and show in Proposition \ref{thm:crossed-product-is-a-Cuntz-algebra}, that this crossed product is isomorphic to the Cuntz algebra $\cO_{n}$ on $n = \dim B$ generators.  We also show in Proposition \ref{prop:Cuntz--Pimsner-is-a-Cuntz-algebra} that the Cuntz--Pimsner algebra $\cO_{E_K}$ is isomorphic to $\cO_n$.  These results combined generalize the well known identifications of Cuntz--Krieger algebras, Cuntz--Pimsner algebras, and Exel crossed products associated to complete graphs.

Finally, in the other subsections of Section~\ref{sec:simpler_examples} of the paper, we study trivial (edgeless) quantum graphs and their two natural generalizations: rank-one quantum graphs, and quantum graphs associated to $\ast$-automorphic quantum adjacency matrices.  We are able  associate a natural choice of (non-commutative) Exel system in these cases as well, except for rank-one quantum graphs. Given one of these types of quantum graphs, we show that the Exel crossed product is isomorphic to the Cuntz--Pimsner algebra for that quantum graph's quantum edge correspondence.  See Corollary \ref{cor:triv}, Proposition \ref{prop:triv}, and Corollary \ref{cor:triv2}. These examples also mimic the classical setting, where a graph's associated Cuntz--Krieger algebra, Exel crossed product, and Cuntz--Pimsner algebra are all isomorphic.

Let us end this introduction with a remark: For more general quantum graphs $\cG = (B,\psi, A)$, it is an interesting and natural problem to associate to $\cG$ a quantum analogue of (functions on) the path space $X_A$. The construction of an appropriate noncommutative version of $C(X_A)$ seems to be highly non-trivial, and we plan to investigate this in more detail in a followup work.

\subsection*{Acknowledgments} Michael Brannan was partially supported by NSF Grant DMS-2000331. Brent Nelson was partially supported by NSF Grant DMS-1856683. Lara Ismert was partially supported by the NSF-AWM Mentoring Travel Grant. Mateusz Wasilewski was partially supported by the Research Foundation — Flanders (FWO) through a Postdoctoral Fellowship, by long term structural funding – Methusalem grant of the Flemish
Government – and by the European Research Council Starting Grant 677120 INDEX.

\section{Preliminaries}
\subsection{Quantum graphs}
In this paper, we consider \emph{finite quantum spaces} $(B,\psi)$, consisting of a finite dimensional C*-algebra $B$ and a distinguished faithful state $\psi\colon B\to \bC$ satisfying
    \[
        mm^* = \delta^2\text{id},
    \]
where $m\colon B\otimes B\to B$ is the multiplication map, $m^*$ is its adjoint with respect to the inner product given by $\psi$, and $\delta>0$. States $\psi:B \to \bC$ satisfying the above identity are called \emph{$\delta$-forms}. Since $B$ is finite dimensional, we have
    \[
        B\cong \bigoplus_{a=1}^d M_{N_a}(\bC).
    \]
The restriction of $\psi$ to the $a^{th}$ summand $M_{N_a}(\mathbb{C})$ appearing in the direct sum decomposition of $B$ will be given by a density matrix with respect to the usual trace, denoted $\rho_a$. We may and will assume that each $\rho_a$ is diagonal with the $i^{th}$ diagonal entry being $\psi(e_{ii}^{a})$. The condition that $\psi$ is a $\delta$-form is then equivalent to $\Tr(\rho_a^{-1})=\sum_{i=1}^{N_a} \psi(e_{ii}^{a})^{-1}=\delta^2$ for each $a=1,\ldots, d$.

We utilize the diagonal entries of these density matrices to define \emph{adapted matrix units} for $(B,\psi)$, given by
    \[
        f_{ij}^{(a)}:= \frac{1}{\left[\psi(e_{ii}^{(a)})\psi(e_{jj}^{(a)})\right]^{1/2}} e_{ij}^{(a)} \qquad \qquad 1\leq a\leq d,\ 1\leq i,j\leq N_a,
    \]
and by \cite[Lemma 3.2]{BEVW20} they satisfy
    \begin{equation}\label{Eq:adaptedunits}
        m^*(f_{ij}^{(a)}) = \sum_{k=1}^{N_a} f_{ik}^{(a)} \otimes f_{kj}^{(a)}.
    \end{equation}
\begin{equation}\label{Eq:prodadaptedunits}
    f_{ij}^{(a)}f_{rs}^{(b)} = \delta_{\substack{a=b\\ j=r}} \frac{1}{\psi(e_{jj}^{(a)})} f_{is}^{(a)}. 
\end{equation}

Given a finite quantum space $(B,\psi)$, a linear map $A\colon B\to B$ is said to be a \emph{quantum adjacency matrix} if
    \[
        m(A\otimes A)m^* = \delta^2 A.
    \]
In this case, the triple $(B,\psi,A)$ is called a \emph{quantum graph}. At times, it will be convenient to express $A$ as an actual matrix $[A_{ija}^{rsb}]_{\substack{1 \le a,b \le m \\
1 \le i,j \le N_a \\
1 \le r,s \le N_b}}$ with respect to the   basis of adapted matrix units:    \[
        A(f_{ij}^{(a)}) = \sum_{b=1}^d \sum_{r,s=1}^{N_b} A_{ija}^{rsb} f_{rs}^{(b)}.
    \]
These coefficients can be used to directly check whether a linear map $A\colon B\to B$ is a quantum adjacency matrix (see \cite[Lemma 3.4]{BEVW20}), but this will not be necessary in the present paper.

\subsection{C*-correspondences and Cuntz--Pimsner algebras}\label{subsec:correspondences_and_CP_algebras}

Given a C*-algebra $B$, a \emph{C*-correspondence} over $B$ is a right Hilbert $B$-module $X$ (with right $B$-valued innner product $\langle \cdot, \cdot \rangle_B$) which admits a $*$-homomorphism $\varphi_X\colon B\to \cL(X)$. Here $\cL(X)$ denotes the right $B$-linear adjointable operators on $X$. The $*$-homomorphism $\varphi_X$ induces a left $B$-action that commutes with the right $B$-action:
    \[
        x\cdot \xi:=\varphi_X(x)\xi \qquad \qquad x\in B,\ \xi\in X.
    \]
One says $X$ is \emph{faithful} if $\varphi_X$ is faithful, and \emph{full} if $\overline{\text{span}}\<X,X\>_B=B$. The \emph{compact operators} on $X$, denoted $\cK(X)$, are generated by $\theta_{\xi,\eta}\in \cL(X)$, $\xi,\eta\in X$, where
    \[
        \theta_{\xi,\eta}(\zeta) := \xi\cdot \<\eta,\zeta\>_B \qquad \zeta\in X.
    \]
In this paper, we will exclusively consider finite dimensional C*-correspondences, for which $\cK(X)=\cL(X)$. The following example is particularly relevant for our purposes.

\begin{example}\label{example:correspondence_from_cp_map}
Let $B$ be a finite dimensional C*-algebra and $A\colon B\to B$ a completely positive map. Define a $B$-valued inner product on $B\otimes B$ by
    \[
        \< a\otimes b, c\otimes d\>_B := b^* A(a^* c)d,
    \]
where the positivity follows from the complete positivity of $A$. Then after taking a separation,
    \[
        B\otimes_A B:= B\otimes B/\{\xi\in B\otimes B\colon \<\xi,\xi\>_B=0\}
    \]
defines a C*-correspondence over $B$. (Since everything is finite-dimensional, we do not need to take a completion.) The usual left and right actions of $B$ on $B\otimes B$, 
    \[
        x\cdot (a\otimes b)\cdot y = (xa)\otimes (by),\qquad a,b,x,y\in B
    \]
extend to left and right actions on $B\otimes_A B$.
\end{example}

After \cite{Kat04}, a \emph{representation} of a C*-correspondence $X$ over $B$ on a C*-algebra $D$ is a pair $(\pi,t)$ consisting of a $*$-homomorphism $\pi\colon B\to D$ and a linear map $t\colon X\to D$ satisfying:
    \begin{enumerate}[label=(\roman*)]
        \item $\pi(x)t(\xi)=t(x\cdot \xi)$ for $x\in B$ and $\xi\in X$,
        
        \item $t(\xi)^*t(\eta)=\pi(\<\xi,\eta\>_B)$ for $\xi,\eta\in X$.
    \end{enumerate}
Using (i), one can also show $t(\xi)\pi(x)= t(\xi\cdot x)$ for $x\in B$ and $\xi\in X$. One can also define a $*$-homomorphism $\psi_t\colon \cK(X)\to D$ by $\psi_t(\theta_{\xi,\eta})=t(\xi)t(\eta)^*$. A representation is said to be \emph{covariant} if $\pi(x) = \psi_t(\varphi_X(x))$ for all $x$ in the \emph{Katsura ideal} $J_X$, which is defined by
    \[
        J_X:=\{x\in B\colon \varphi_X(x)\in \cK(X) \text{ and } xy=0 \text{ for all }y\in \ker{\varphi_X}\}.
    \]
Note that $J_X=B$ when $X$ is finite dimensional and faithful, which is the situation we will primarily consider. 

The \emph{Cuntz--Pimsner algebra} for a C*-correspondence $X$ over $B$ is the C*-algebra $\cO_X=C^*(\pi_X(B),t_X(X))$ where $(\pi_X,t_X)$ is the universal covariant representation of $X$. That is, given any covariant representation $(\pi,t)$ of $X$ on a C*-algebra $D$, there exists a $*$-homomorphism $\rho \colon \cO_X\to D$ satisfying $\pi=\rho\circ \pi_X$ and $t=\rho\circ t_X$.

\section{Quantum edge correspondences}\label{sec:quantumedgecorrespondences}

Let $(B,\psi)$ be a finite quantum space.  Since the state $\psi:B \to \bC \subset B$ is completely positive we can consider the C*-correspondence $B\otimes_\psi B$ over $B$ from Example~\ref{example:correspondence_from_cp_map}. Note that for $\xi\in B\otimes B$, $\psi(\<\xi,\xi\>_B)=\|\xi\|^2_{\psi\otimes \psi}$, and so $B\otimes_\psi B=B\otimes B$ as a vector space. That is, we do not need to take a quotient of $B\otimes B$.

\begin{definition}\label{def:edgecorrespondence}
Let $\cG=(B,\psi,A)$ be a directed quantum graph such that $\psi$ is a $\delta$-form. We define the {\it quantum edge indicator} to be the element \[\epsilon_\cG:=\frac{1}{\delta^2}(1\otimes A)m^*(1) \in B\otimes_\psi B.\] The \textit{quantum edge correspondence} of $\cG$ is the C*-correspondence over $B$ defined by
    \[
        E_\cG:=B\cdot \epsilon_\cG \cdot B =\text{span}\{x\cdot \epsilon_\cG\cdot y \colon x,y\in B\} \subset B\otimes_\psi B.
    \]
That is, $E_\cG$ is the C*-subcorrespondence of $B\otimes_\psi B$ generated by  the quantum edge indicator $\epsilon_\cG$.
\end{definition}

\begin{example}\label{exmp:commutativecase}
    In the case of a classical directed graph $\mathcal{G}=( \bC(V), \frac{1}{|V|}, A)$, one has
    \[
        \epsilon_\cG=\frac{1}{\delta^2}(1\otimes A)m^*(1) = \frac{1}{|V|}(1\otimes A)\sum_{v\in V} |V| p_v\otimes p_v = \sum_{w\to v} p_v\otimes p_w.
    \]
Hence $\epsilon_\cG\in \bC(V\times V)$ is the indicator function for the set $\{(v,w)\colon (w,v) \text{ is an edge}\}$, and the quantum edge correspondence $E_\cG$ is the space of functions supported on this set.
\end{example}

\subsection{Properties of quantum edge indicators and correspondences}

Before studying the edge correspondence, we note some important properties of the quantum edge indicator $\epsilon_\cG$ which further justify our terminology. 

\begin{proposition}\label{prop:properties_of_edge_generator}
Let $\mathcal{G}=(B,\psi,A)$ be a directed quantum graph with $\delta$-form $\psi$, and let $\epsilon_\cG:= \frac{1}{\delta^2}(1\otimes A)m^*(1)$ be the quantum edge indicator. Then:
    \begin{enumerate}[label=(\arabic*)]
        \item $A(x) = \delta^2 (\psi\otimes 1)(x\cdot \epsilon_\cG)$ for all $x\in B$.
    
        \item $\epsilon_\cG\# \epsilon_\cG=\epsilon_\cG$ where $(a\otimes b)\#(c\otimes d)=(ac)\otimes (db)$ for $a,b,c,d\in B$.
        
        \item $A$ is completely positive if and only if $(\sigma_{i/2}^\psi\otimes 1)(\epsilon_\cG)$ is self-adjoint, where $(\sigma_t^\psi)_{t \in \mathbb R}$ denotes the modular automorphism group of $\psi$.
    \end{enumerate}
\end{proposition}
\begin{proof}
\begin{enumerate}
    \item[\textbf{(1):}] For $x,y\in B$, we have
    \begin{align*}
        \< \delta^2 (\psi\otimes 1)(x\cdot \epsilon_\cG), y\>_\psi &= \delta^2 \< \epsilon_\cG, x^*\otimes y\>_{\psi\otimes \psi}\\
            &= \< m^*(1), x^* \otimes A^*(y)\>_{\psi\otimes \psi}\\
            &= \< 1, x^* A^*(y)\>_\psi\\
            &= \< A(x), y\>_\psi.
    \end{align*}
Hence $\delta^2 (\psi\otimes 1)(x\cdot \epsilon_\cG)=A(x)$ as claimed.
    
    \item[\textbf{(2):}] Observe that $\epsilon_\cG$ is the image of $1\otimes 1$ under the following map on $B\otimes B$:
        \[
            \frac{1}{\delta^2}(1\otimes m)(1\otimes A\otimes 1)(m^*\otimes 1).
        \]
    We will show that this map is a left and right $B$-linear and idempotent.  (In fact, $E_\cG$ is precisely the image of $B\otimes B$ under this map.) 
    Bimodularity comes from the fact that both $m$ and $m^*$ are left and right $B$-linear. Using the associativity of the multiplication operation (i.e. $m(m\otimes 1) = m(1\otimes m)$, hence also $(m^*\otimes 1)m^* = (1\otimes m^*)m^*$) and that $A$ is a quantum adjacency matrix (i.e. $\frac{1}{\delta^2} m(A\otimes A)m^*=A$), we will show that this map is idempotent. Its square is equal to 
    \[
    \frac{1}{\delta^4}     \left(1\otimes (m(1\otimes m)\right)  \left(1\otimes A\otimes A \otimes 1\right)\left((m^*\otimes 1)m^* \otimes 1\right).
    \] 
    Using associativity, we get
    \[
    \frac{1}{\delta^4}(1\otimes m)(1\otimes m \otimes 1)(1\otimes A \otimes A \otimes 1)(1\otimes m^* \otimes 1)(m^* \otimes 1).
    \]
    In the middle, we recognize the expression $1\otimes (m(A\otimes A)m^*)\otimes 1$, which is equal to $\delta^2 1\otimes A \otimes 1$. In the end we obtain
    \[
    \frac{1}{\delta^2}(1\otimes m)(1\otimes A \otimes 1)(m^*\otimes 1),
    \]
    which verifies the map is idempotent. Consequently,
        \begin{align*}
            \epsilon_\cG\# \epsilon_\cG &=  \epsilon_\cG\# \frac{1}{\delta^2} (1\otimes m)(1\otimes A\otimes 1)(m^*\otimes 1)(1\otimes 1)\\
                &= \frac{1}{\delta^2}(1\otimes m)(1\otimes A\otimes 1)(m^*\otimes 1)(\epsilon_\cG)\\
                &= \left[\frac{1}{\delta^2}(1\otimes m)(1\otimes A\otimes 1)(m^*\otimes 1)\right]^2(1\otimes 1)\\
                &= \frac{1}{\delta^2}(1\otimes m)(1\otimes A\otimes 1)(m^*\otimes 1)(1\otimes 1) = \epsilon_\cG.
        \end{align*}
    (This identity can also be checked directly using the adapted matrix units of $(B,\psi)$, but this arduous task is left to the skeptical reader.)
    
    \item[\textbf{(3):}] Let $\eta_\cG\in B\otimes B^\text{op}$ be the image of $\epsilon_\cG$ under the map $a\otimes b\mapsto a\otimes b^\circ$. Then $(\sigma_{i/2}^\psi\otimes 1)(\epsilon_\cG)$ is self-adjoint if and only if $(\sigma_{i/2}^\psi\otimes 1)(\eta_\cG)$ is self-adjoint. Observe that $(\sigma_{i/2}^\psi\otimes 1)(\eta_\cG)$ is an idempotent by part (2), and so it is self-adjoint if and only if it is a projection, and hence if and only if it is positive. It therefore suffices to show $A$ is completely positive if and only if $(\sigma_{i/2}^\psi\otimes 1)(\eta_\cG)$ is positive. 
    
    Recall that if $(M,\varphi)$ is a von Neumann algebra equipped with a faithful normal linear functional, then $\overline{\{\Delta_\varphi^{1/4} x\colon x\in M_+\}}$ is a self-dual cone in $L^2(M,\varphi)$, where $\Delta_\varphi$ is the modular operator with respect to $\varphi$. Consequently $A$ is completely positive if and only if for all positive elements $X=(x_{ij}), Y=(y_{ij})\in M_n(B)^+$ one has
        \[
            \< \Delta_{\psi\otimes \text{Tr}}^{1/4} (A\otimes I_n)(X), \Delta_{\psi\otimes \text{Tr}}^{1/4}Y\>_{\psi\otimes \text{Tr}}= \sum_{i,j=1}^n \< \Delta_\psi^{1/4} A(x_{ij}), \Delta_\psi^{1/4} y_{ij}\>_\psi \geq 0.
        \]
    Using part (1), we compute
        \[
            \< \Delta_\psi^{1/4} A(x_{ij}), \Delta_\psi^{1/4} y_{ij}\>_\psi = \< \delta^2(\psi\otimes 1)(x_{ij}\cdot \epsilon_\cG), \Delta_\psi^{1/2} y_{ij}\>_\psi = \delta^2 \< \epsilon_\cG, x_{ij}^*\otimes (\Delta_\psi^{1/2} y_{ij})\>_{\psi\otimes \psi}.
        \]    
Now, $L^2(B\otimes B,\psi\otimes\psi)\ni a\otimes b\mapsto a\otimes [\Delta_\psi^{-1/2} b]^\circ \in L^2(B\otimes B^{\text{op}},\psi\otimes\psi^{\text{op}})$ is an isometry, so we can continue the above computation with
        \begin{align*}
             \< \Delta_\psi^{1/4} A(x_{ij}), \Delta_\psi^{1/4} y_{ij}\>_\psi &= \delta^2 \< (1\otimes \Delta_\psi^{-1/2}) \eta_\cG, x_{ij}^*\otimes y_{ij}^\circ\>_{\psi\otimes \psi^{\text{op}}}\\
             &= \delta^2\< (\Delta_\psi^{1/4}\otimes \Delta_\psi^{-1/4})( \Delta_\psi^{-1/2}\otimes 1)\eta_\cG, (\Delta_\psi^{1/4}\otimes \Delta_\psi^{-1/4})( x_{ij}^*\otimes y_{ij}^\circ) \>_{\psi\otimes \psi^\text{op}}\\
            &= \delta^2 \< \Delta_{\psi\otimes \psi^{\text{op}}}^{1/4}\left[ (\sigma_{i/2}^{\psi}\otimes 1)(\eta_\cG)\right], \Delta_{\psi\otimes \psi^{\text{op}}}^{1/4}(x_{ij}^*\otimes y_{ij}^\circ)\>_{\psi\otimes \psi^{\text{op}}}.
        \end{align*}
Thus
    \begin{align}\label{eqn:complete_positivity_equation}
        \sum_{i,j=1}^n \< \Delta_\psi^{1/4} A(x_{ij}), \Delta_\psi^{1/4} y_{ij}\>_\psi & \nonumber\\
            = \delta^2 \sum_{i,j=1}^n &\< \Delta_{\psi\otimes \psi^{\text{op}}}^{1/4}\left[ (\sigma_{i/2}^{\psi}\otimes 1)(\eta_\cG)\right], \Delta_{\psi\otimes \psi^{\text{op}}}^{1/4}(x_{ij}^*\otimes y_{ij}^\circ)\>_{\psi\otimes \psi^{\text{op}}}.
    \end{align}
Suppose $X=W^*W$ and $Y=Z^*Z$ for $W=(w_{ij}), Z=(z_{ij})\in M_n(B)$. Then
    \[
        \sum_{i,j=1}^n x_{ij}^*\otimes y_{ij}^\circ = \sum_{i,j,k,\ell=1}^n w_{kj}^*w_{ki}\otimes (z_{\ell i}^* z_{\ell j})^\circ = \sum_{k,\ell=1}^n (\sum_{j=1}^n w_{kj}\otimes (z_{\ell j}^*)^\circ)^* (\sum_{i=1}^n w_{ki}\otimes (z_{\ell i}^*)^\circ)
    \]
is positive. Also note that every positive element of $B\otimes B^\text{op}$ can be presented this way. Indeed,
    \[
        (\sum_{d=1}^m w_d\otimes z_d^\circ)^*(\sum_{d=1}^m w_d\otimes z_d^\circ) = \sum_{i,j=1}^n x_{ij}^*\otimes y_{ij}^\circ
    \] 
if $X=W^*W, Y=Z^*Z\in M_m(B)$ where $(W)_{ij} = \delta_{i=1} w_j$ and $(Z)_{ij}=\delta_{i=1} z_i^*$. It follows from (\ref{eqn:complete_positivity_equation}) and these observations that $A$ is completely positive if and only if $(\sigma_{i/2}^{\psi}\otimes 1)(\eta_\cG)$ is positive.
\end{enumerate}
\end{proof}

\begin{remark}
At this point it is worth remarking on the connection to Weaver's notion of a quantum graph \cite{We12, We21}. Recall that there is an algebra isomorphism 
    \begin{align*}
        \pi:B \otimes B^{\op} &\to {}_{B'}CB_{B'}(B(L^2(B,\psi)))\\
        (a\otimes b^\circ)&\mapsto (T\mapsto aTb) 
    \end{align*} 
    where ${}_{B'}CB_{B'}(B(L^2(B,\psi)))$ denotes the space of completely bounded $B'$-$B'$-bimodule maps on $B(L^2(B,\psi))$.   When a quantum adjacency matrix $A$ is completely positive, the projection $p = (\sigma_{i/2}^\psi\otimes 1)(\eta_\cG) \in B\otimes B^\op$ induces a  $B'$-$B'$-bimodule projection $\pi(p)$ on $B(L^2(B, \psi))$ whose range $S \subseteq B(L^2(B,\psi))$ is a $B'$-$B'$-bimodule. Such bimodules $S$ are exactly what Weaver refers to as (directed, non-reflexive) quantum graphs on $B \subseteq B(L^2(B,\psi))$.
\end{remark}

\begin{theorem}\label{thm:cp_Adjacency_matrices_are_inner_products}
Let $\cG=(B,\psi, A)$ be a directed quantum graph such that $\psi$ is a $\delta$-form and $A$ is completely positive. Then for $x,y\in B$ one has
    \[
        \< x \cdot \epsilon_\cG, y\cdot \epsilon_\cG\>_B = \frac{1}{\delta^2}A(x^*y),
    \]
where $\epsilon_\cG = \frac{1}{\delta^2} (1\otimes A)m^*(1)$ is the quantum edge indicator. In particular, $A(x)=\delta^2\<\epsilon_\cG, x\cdot \epsilon_\cG\>_B$ for all $x\in B$.
\end{theorem}
\begin{proof}
Write $\epsilon_\cG = \sum_\alpha p_\alpha\otimes q_\alpha$ for $p_\alpha,q_\alpha\in B$. Then
    \begin{align*}
         \< x \cdot \epsilon_\cG, y \cdot \epsilon_\cG\>_B &= \sum_{\alpha,\beta} \< xp_\alpha\otimes q_\alpha, yp_\beta\otimes q_\beta\>_B\\
            &= \sum_{\alpha, \beta} \psi( p_\alpha^* x^* y q_\beta) q_\alpha^* q_\beta\\
            &=  \sum_{\alpha, \beta} \psi(x^*y p_\beta \sigma_{-i}^\psi(p_\alpha^*)) q_\alpha^* q_\beta \\
            &=  (\psi\otimes 1)[ x^*y\cdot \epsilon_\cG \# (\sigma_{-i}^\psi\otimes 1)(\epsilon_\cG^*)]\\
            &=\ (\psi\otimes 1)( x^*y\cdot \epsilon_\cG \# \epsilon_\cG)\\
            &= (\psi\otimes 1)( x^*y\cdot \epsilon_\cG) =  \frac{1}{\delta^2} A(x^*y),
    \end{align*}
where the last three equalities follow from the three parts of  Proposition~\ref{prop:properties_of_edge_generator} (in reverse order).
\end{proof} 

Recalling the definition of $B\otimes_A B$ from Example~\ref{example:correspondence_from_cp_map}, we obtain the following corollary:

\begin{corollary}\label{cor:quantum_graph_correspondence_as_cp_correspondence}
Let $\cG=(B,\psi, A)$ be a directed quantum graph such that $\psi$ is a $\delta$-form and $A$ is completely positive. Then $E_\cG\cong B\otimes_A B$ as C*-correspondences over $B$ via the map $x\cdot \epsilon_\cG\cdot y\mapsto \frac{1}{\delta}(x\otimes y)$.
\end{corollary}

The above corollary is somewhat surprising in that a priori $B\otimes_A B$ appears to be completely independent of $\psi$. But of course the dependence is hidden in the fact that $A$ is a quantum adjacency matrix for $(B,\psi)$. We also remark that while $B\otimes_A B$ is more easily defined, $E_\cG$ has the advantage of not requiring us to quotient by any null spaces.

Given a finite quantum space $(B,\psi)$ such that $\psi$ is a $\delta$-form, Theorem~\ref{thm:cp_Adjacency_matrices_are_inner_products} tells us every completely positive adjacency matrix is of the form $A(x)=\delta^2\<\xi, x\cdot \xi\>_B$ for some $\xi\in B\otimes_\psi B$. Thus it is natural to ask for which elements $\xi\in B\otimes_\psi B$ is the map $B\ni x\mapsto \delta^2\< \xi,x\cdot \xi\>_B$ a (necessarily completely positive) quantum adjacency matrix? Suppose $\xi\#\xi=\xi$ and $(\sigma_{i/2}^\psi\otimes 1)(\xi)$ is self-adjoint, and define $A_\xi(x):= \delta^2 \<\xi,x\cdot \xi\>_B$, then the same computation as in the above theorem implies
    \[
        A_\xi(x)=\delta^2\< \xi, x\cdot \xi\>_B = \delta^2(\psi\otimes 1)(x\cdot \xi).
    \]
Consequently, for each $1\leq a \leq d$ and $1\leq i,j\leq N_a$
    \begin{align*}
        \frac{1}{\delta^2} m(A_\xi\otimes A_\xi)m^*(f_{ij}^{(a)}) &= \frac{1}{\delta^2}\sum_{k} A_\xi( f_{ik}^{(a)}) A_\xi(f_{kj}^{(a)})\\
        &= \delta^2 \sum_k (\psi\otimes 1)(f_{ik}^{(a)}\cdot \xi) (\psi\otimes 1)(f_{kj}^{(a)}\cdot \xi).
    \end{align*}
Writing $\xi= \sum_\alpha p_\alpha\otimes q_\alpha$, we can continue the above with
    \begin{align*}
        \frac{1}{\delta^2} m(A_\xi\otimes A_\xi)m^*(f_{ij}^{(a)}) &= \delta^2 \sum_{k,\alpha,\beta} \<f_{ki}^{(a)}, p_\alpha\>_\psi\<f_{jk}^{(a)}, p_\beta\>_\psi q_\alpha q_\beta\\
        &=\delta^2 \sum_{k,\alpha,\beta} \<f_{ki}^{(a)}, p_\alpha\>_\psi\<f_{ik}^{(a)}, \psi(e_{ii}^{(a)}) f_{ij}^{(a)} p_\beta\>_\psi q_\alpha q_\beta\\
        &= \delta^2 \sum_{k,\alpha,\beta} \< \<\psi(e_{ii}^{(a)}) f_{ij}^{(a)} p_\beta, f_{ik}^{(a)}\>_\psi f_{ki}^{(a)}, p_\alpha\>_\psi  q_\alpha q_\beta\\
        &= \delta^2 \sum_{k,\alpha,\beta} \< \< \Delta_\psi f_{ki}^{(a)}, \psi(e_{ii}^{(a)}) p_\beta^* f_{ji}^{(a)}\>_\psi f_{ki}^{(a)}, p_\alpha\>_\psi  q_\alpha q_\beta\\
        &= \delta^2 \sum_{k,\alpha,\beta} \psi(e_{kk}^{(a)}) \< \< f_{ki}^{(a)}, p_\beta^*f_{ji}^{(a)}\>_\psi f_{ki}^{(a)}, p_\alpha\>_\psi  q_\alpha q_\beta.
    \end{align*}
Now, using the fact that $\{ \psi(e_{rr}^{(b)})^{1/2}f_{rs}^{(b)}\colon 1\leq b \leq d,\ 1\leq r,s\leq N_b \}$ is an orthonormal basis for $L^2(B)$, we then have
    \begin{align*}
       \frac{1}{\delta^2} m(A_\xi\otimes A_\xi)m^*(f_{ij}^{(a)}) &= \delta^2 \sum_{r,s,b,\alpha,\beta} \< \< \psi(e_{rr}^{(b)})^{1/2}f_{rs}^{(b)}, p_\beta^*f_{ji}^{(a)}\>_\psi \psi(e_{rr}^{(b)})^{1/2}f_{rs}^{(b)}, p_\alpha\>_\psi  q_\alpha q_\beta\\
        &= \delta^2 \sum_{\alpha, \beta} \<p_\beta^* f_{ji}^{(a)}, p_\alpha\>_\psi q_\alpha q_\beta\\
        &= \delta^2 \sum_{\alpha,\beta } (\psi\otimes 1)( f_{ij}^{(a)}\cdot (p_\beta \otimes q_\beta)\# (p_\alpha\otimes q_\alpha))\\
        &= \delta^2 \sum_\alpha (\psi\otimes 1)( f_{ij}^{(a)}\cdot (p_\alpha \otimes q_\alpha)) = A_\xi(f_{ij}^{(a)}).
    \end{align*}
Thus $A_\xi$ is a quantum adjacency matrix for $(B,\psi)$. Moreover, by a similar computation one can show that $\frac{1}{\delta^2}(1\otimes A_\xi)m^*(1)=\xi$ and so there is in fact a one-to-one correspondence between completely positive quantum adjacency matrices for $(B,\psi)$ and the set
    \[
        \{\xi\in B\otimes_\psi B\colon  \xi\#\xi=\xi,\ (\sigma_{i/2}^\psi\otimes 1)(\xi)^*=(\sigma_{i/2}^\psi\otimes 1)(\xi)\}.
    \]
More generally, we have the following recognition result for the quantum edge correspondence.

\begin{proposition}\label{prop:recognition}
Let $(B,\psi)$ be a finite quantum space such that $\psi$ is a $\delta$-form, and let $X=B\cdot\xi\cdot B$ be a cyclic C*-correspondence over $B$. If the map $A\colon B\to B$ defined by $A(x)=\delta^2\<\xi, x\cdot \xi\>_B$ is a quantum adjacency matrix for $(B,\psi)$, then $x\cdot \xi\cdot y\mapsto x\cdot  \epsilon_{(B,\psi,A)}\cdot y$ extends to a C*-correspondence isomorphism from $X$ to the quantum edge correspondence $E_{(B,\psi, A)}$.
\end{proposition}
\begin{proof}
If we denote $\epsilon:=\epsilon_{(B,\psi,A)}$, then Theorem~\ref{thm:cp_Adjacency_matrices_are_inner_products} implies
    \[
        \<e_{ij}^{(a)}\cdot \xi, e_{rs}^{(b)}\cdot \xi\>_B = \delta_{\substack{a=b\\i=r}} \<\xi, e_{js}^{(a)}\cdot \xi\>_B = \delta_{\substack{a=b\\i=r}} \frac{1}{\delta^2}A(e_{js}^{(a)}) = \< e_{ij}^{(a)}\cdot\epsilon, e_{rs}^{(b)}\cdot \epsilon\>_B
    \]
for all $1\leq a, b\leq d$, $1\leq i,j\leq N_a$, and $1\leq r,s\leq N_b$. It follows that for all $x,y,z,w\in B$ that
    \[
        \<x\cdot \xi\cdot y, z\cdot \xi\cdot w\>_B = y^* \< x\cdot \xi, z\cdot \xi\>_B w = y^*\< x\cdot \epsilon, z\cdot \epsilon\>_B w = \< x\cdot \epsilon\cdot y, z\cdot \epsilon\cdot w\>_B.
    \]
Therefore $\sum_i x_i\cdot \xi\cdot y_i\mapsto \sum_i x_i\cdot\epsilon\cdot y_i$ is a well-defined, inner product preserving (and hence injective), $B$-bilinear map from $B\cdot\xi\cdot B=X$ onto $B\cdot \epsilon\cdot B =E_\cG$. In other words, it is a C*-correspondence isomorphism.
\end{proof}

It is straightforward to check that any homomorphism $A:B\to B$ is a quantum adjacency matrix for a quantum space $(B,\psi)$ (independent, in fact, of $\psi$). Furthermore, every homomorphism of a C*-algebra that is completely positive must necessarily be $*$-preserving, so it is natural to ask if the one-to-one correspondence in Proposition~\ref{prop:recognition} restricts to a one-to-one correspondence between $*$-homomorphisms and a subset of $\{\xi\in B\otimes_\psi B\colon  \xi\#\xi=\xi,\ (\sigma_{i/2}^\psi\otimes 1)(\xi)^*=(\sigma_{i/2}^\psi\otimes 1)(\xi)\}.$ 

\begin{proposition}\label{prop:when_A_is_a_homomorphism}
Let $(B,\psi)$ be a finite quantum space, and suppose $A:B\to B$ is a completely positive quantum adjacency matrix for $(B,\psi).$ The following are equivalent:
\begin{enumerate}[label=(\roman*)]
\item $A$ is a homomorphism.
\item For all $x,y\in B$, we have $(xy)\cdot \epsilon_\cG=x\cdot \epsilon_\cG\cdot A(y)$.
\end{enumerate}
\end{proposition}

\begin{proof}
Suppose for all $x,y\in B$, we have $(xy)\cdot \epsilon_\cG=x\cdot \epsilon_\cG\cdot A(y).$ Proposition~\ref{prop:recognition} implies
$$
    A(xy)
    =\langle \epsilon_\cG, (xy)\cdot \epsilon_\cG\rangle_B
    =\langle \epsilon_\cG, x\cdot \epsilon_\cG\cdot A(y)\rangle_B
     =\langle \epsilon_\cG, x\cdot \epsilon_\cG\rangle_B A(y)
     =A(x)A(y)
$$
for all $x,y\in B.$ Conversely, suppose $A$ is a homomorphism, and fix $x,y\in B$. For any $a,b\in B$,
$$
\langle a\cdot \epsilon_\cG \cdot b, (xy)\cdot \epsilon_\cG\rangle_B
=
b^* A(a^*xy)%
=
b^* A(a^*x)A(y)
=
\langle a\cdot\epsilon_\cG\cdot b, x\cdot \epsilon_\cG\cdot  A(y)\rangle_B
$$
As elements of the form $a\cdot \epsilon_\cG\cdot b$ span $E_\cG$, for any $\xi\in E_\cG$, 
$$\langle \xi, (xy)\cdot \epsilon_\cG - x\cdot \epsilon_\cG \cdot A(y)\rangle_B=0\quad \text{for all}\;x,y\in B.$$ In particular, $\xi_o:=(xy)\cdot \epsilon_\cG - x\cdot \epsilon_\cG \cdot A(y)$ is an element of $E_\cG$, so $\langle \xi_o,\xi_o\rangle_B=0$. By positive-definiteness of $\langle \cdot, \cdot \rangle_B,$ we may conclude $\xi_o=0.$ Thus, $xy\cdot \epsilon_\cG=x\cdot \epsilon_\cG\cdot A(y)$ for all $x,y\in B.$
\end{proof}

When $A$ is a $*$-automorphism, the quantum edge correspondence $E_\cG$ arising from a quantum graph $\cG:=(B,\psi,A)$ is the span of $B\cdot \epsilon_\cG$ or the span of $\epsilon_\cG\cdot B$. Indeed, given $x\cdot \epsilon_\cG\cdot y\in E_\cG$ for some $x,y\in B$, Proposition~\ref{prop:when_A_is_a_homomorphism} implies that the maps $E_\cG\to E_\cG$ given by $x\cdot \epsilon_G\cdot y\mapsto xA^{-1}(y)\cdot \epsilon_\cG$ and  $x\cdot \epsilon_\cG\cdot y\mapsto \epsilon_\cG\cdot A(x)y$ are both the identity map on $E_\cG.$

\subsection{Faithfulness and fullness of the quantum edge correspondence}

Recall the definitions of faithful and full for C*-correspondences from Section~\ref{subsec:correspondences_and_CP_algebras}. The next theorem shows that these features for a quantum edge correspondence $E_\cG$ are determined by the quantum adjacency matrix $A$.

\begin{theorem}\label{thm:faithful_full_correspondence}
Let $\cG=(B,\psi,A)$ be a directed quantum graph such that $\psi$ is a $\delta$-form and $A$ is completely positive, and let $E_\cG$ be the quantum edge correspondence. If $1_a:=\sum_{i=1}^{N_a} e_{ii}^{(a)}$ for each $1\leq a\leq d$, then
    \[
        \{x\in B\colon x\cdot \xi=0\ \forall \xi\in E_\cG\}= \left(BA^{\ast}(B)B\right)^{\perp}
    \]
and
    \[
        \overline{\text{span}}\<E_\cG, E_\cG\>_B = 
        B\cdot A(B) \cdot B,
    \]
    the two-sided ideal of $B$ generated by the range of $A$. 
In particular, $E_\cG$ is faithful if and only if  $\ker(A)$ does not contain a central summand of $B$, and $E_\cG$ is full if and only if $A(B)$ is not orthogonal to a central summand of $B$.
\end{theorem}
\begin{proof}
Suppose $x\cdot \xi=0$ for all $\xi\in E_\cG$. This means precisely that for all $a,b \in B$ we have $xa\epsilon_{\cG} b=0$, which, on the other hand, is equivalent to $\langle c \epsilon_{\cG} d, xa \epsilon_{\cG} b\rangle=0$ for all $c,d \in B$. This expression is equal to $\delta^{-2}d^* A(c^* xa) b=0$, so it is equal to zero if and only if $A(c^* x a)=0$. This time using the inner product on $B$, we see that this is equivalent to $0=\langle y, A(c^* x a)\rangle = \langle c A^{\ast}(x) a^*, x \rangle $ for all $y \in B$. Therefore $x \in \left(B A^{\ast}(B) B\right)^{\perp}$.

Next, observe that
\begin{align*}
    \overline{\text{span}}\langle E_\cG, E_\cG\rangle_B &= \overline{\text{span}}\big\{\langle a \cdot \epsilon_\cG \cdot b, c \cdot \epsilon_\cG \cdot d\rangle_B: a,b,c,d \in B\big\} \\
    &= \overline{\text{span}}\big\{\delta^{-2}b^*A(a^*c)d: a,b,c,d \in B\big\} \\
    &=B \cdot A(B) \cdot B 
\end{align*}

\end{proof}

\begin{remark}
In the case of a classical directed graph $\cG=( C(V), \psi, A)$, the central summands of $C(V)$ are indexed by $V$. A central summand belongs to $\ker(A)$ when the corresponding vertex is a source (i.e. has no edges into it), and it is orthogonal to $A(C(V))$ when the corresponding vertex is a sink (i.e. has no edges out of it). Hence $E_\cG$ is faithful when $A$ has no zero columns and is full when $A$ has no zero rows.
\end{remark}

The previous remark motivates the following definitions.

\begin{definition}\label{def:quantum-sink-and-source}
 Let $\cG=(B,\psi, A)$ be a directed quantum graph such that $\psi$ is a $\delta$-form and $A$ is completely positive. A \textit{quantum sink} in $\cG$ is a central summand of $B$ that is orthogonal to the range of $A$. A \textit{quantum source} in $\cG$ is a central summand of $B$ that lies in the kernel of $A$.
\end{definition}

In the next section we will examine the Cuntz--Pimsner algebra $\cO_{E_\cG}$ associated to the quantum edge correspondence. Consequently, it is important to understand how to express the left $B$-action in terms of compact operators $\cK(E_\cG)$.

\begin{theorem}\label{thm:compact_operators}
Let $\cG=(B,\psi,A)$ be a directed quantum graph such that $\psi$ is a $\delta$-form and $A$ is completely positive, and let $E_\cG=B\cdot \epsilon_\cG\cdot B$ be the quantum edge correspondence. If $\{ f_{ij}^{(a)}\colon 1\leq a\leq d,\ 1\leq i,j\leq N_a\}$ are the adapted matrix units for $(B,\psi)$, then
    \[
        f_{ij}^{(a)}\cdot \xi = \sum_{k=1}^{N_a} \theta_{f_{ik}^{(a)}\cdot \epsilon_\cG,f_{jk}^{(a)}\cdot \epsilon_\cG}(\xi)
        \quad
        \text{for all }
        \xi\in E_\cG.
    \]

\end{theorem}
\begin{proof}
Since both sides are right $B$-linear in $\xi$, it suffices to prove the equality for $\xi=e_{rs}^{(b)}\cdot \epsilon_\cG$, $1\leq b\leq d$ and $1\leq r,s\leq N_b$. Write $\epsilon_\cG=\sum_\alpha x_\alpha\otimes y_\alpha$. Then using Proposition~\ref{prop:properties_of_edge_generator}.(2) we have
    \begin{align*}
        f_{ij}^{(a)}\cdot (e_{rs}^{(b)}\cdot \epsilon_\cG) &= \delta_{\substack{a=b\\ j=r}} [\psi(e_{ii}^{(a)})\psi(e_{jj}^{(a)})]^{-1/2} e_{is}^{(b)}\cdot \epsilon_\cG\\
        &= \delta_{\substack{a=b\\ j=r}} [\psi(e_{ii}^{(a)})\psi(e_{jj}^{(a)})]^{-1/2} e_{is}^{(b)}\cdot (\epsilon_\cG\# \epsilon_\cG)\\
        &= \delta_{\substack{a=b\\ j=r}} [\psi(e_{ii}^{(a)})\psi(e_{jj}^{(a)})]^{-1/2} \sum_{k,\alpha,\beta} (e_{is}^{(b)}x_\alpha e_{kk}^{(a)} x_\beta)\otimes (y_\beta y_\alpha)\\
        &= \delta_{\substack{a=b\\ j=r}} [\psi(e_{ii}^{(a)})\psi(e_{jj}^{(a)})]^{-1/2} \sum_{k,\alpha,\beta} (x_\alpha)_{sk}^{(a)} (e_{ik}^{(a)}x_\beta)\otimes (y_\beta y_\alpha)\\
        &= \delta_{\substack{a=b\\ j=r}} [\psi(e_{ii}^{(a)})\psi(e_{jj}^{(a)})]^{-1/2} \sum_{k,\alpha,\beta} \frac{1}{\psi(e_{kk}^{(a)})} \psi( e_{ks}^{(a)}x_\alpha)(e_{ik}^{(a)}x_\beta)\otimes (y_\beta y_\alpha)\\
        &= \delta_{\substack{a=b\\ j=r}} [\psi(e_{ii}^{(a)})\psi(e_{jj}^{(a)})]^{-1/2} \sum_{k} \frac{1}{\psi(e_{kk}^{(a)})} e_{ik}^{(a)}\cdot \epsilon_\cG \cdot (\psi\otimes 1)(e_{ks}^{(a)}\cdot \epsilon_\cG).
    \end{align*}
Now, using Proposition~\ref{prop:properties_of_edge_generator}.(1) and Theorem~\ref{thm:cp_Adjacency_matrices_are_inner_products} we see that
    \[
        (\psi\otimes 1)(e_{ks}^{(a)}\cdot \epsilon_\cG) = \frac{1}{\delta^2} A(e_{ks}^{(a)}) = \< \epsilon_\cG, e_{ks}^{(a)}\cdot \epsilon_\cG\>_B = \<  e_{jk}^{(a)}\cdot \epsilon_\cG , e_{js}^{(a)}\cdot \epsilon_\cG\>_B.
    \]
So continuing our computation above, we have
    \begin{align*}
        f_{ij}^{(a)}\cdot (e_{rs}^{(b)}\cdot \epsilon_\cG) &= \delta_{\substack{a=b\\ j=r}} [\psi(e_{ii}^{(a)})\psi(e_{jj}^{(a)})]^{-1/2} \sum_{k} \frac{1}{\psi(e_{kk}^{(a)})} e_{ik}^{(a)}\cdot \epsilon_\cG \cdot \<  e_{jk}^{(a)}\cdot \epsilon_\cG , e_{js}^{(a)}\cdot \epsilon_\cG\>_B\\
            &= \sum_{k} f_{ik}^{(a)}\cdot \epsilon_\cG \cdot \< f_{jk}^{(a)}\cdot \epsilon_\cG, e_{rs}^{(b)}\cdot \epsilon_\cG\>_B\\
            &= \sum_k \theta_{f_{ik}^{(a)}\cdot \epsilon_\cG, f_{jk}^{(a)}\cdot \epsilon_\cG}( e_{rs}^{(b)}\cdot \epsilon_\cG),
    \end{align*}
as claimed.
\end{proof}

The following corollary is a rephrasing of Theorems~\ref{thm:cp_Adjacency_matrices_are_inner_products} and \ref{thm:compact_operators} in terms of linear maps, which will be useful in the next section. We leave the proof to the reader.

\begin{corollary}\label{cor:abstract_Toeplitz}
Let $\cG=(B,\psi,A)$ be a directed quantum graph such that $\psi$ is a $\delta$-form and $A$ is completely positive, and let $E_\cG=B\cdot \epsilon_\cG\cdot B$ be the quantum edge correspondence. If $(\pi,t)$ is a covariant representation of $E_\cG$ on C*-algebra $D$, then the linear map $T\colon B\to D$ defined by $T(x):= t(x\cdot \epsilon_\cG)$ satisfies
    \begin{align*}
        \mu_D(T^*\otimes T) &= \frac{1}{\delta^2}\pi A m,\\
        \mu_D(T\otimes T^*) m^* &=\psi_t,
    \end{align*}
where $\mu_D\colon D\otimes D\to D$ is the multiplication map, $T^*(x):=T(x^*)^*$, and $\psi_t\colon \cK(E_\cG)\to D$ is the $*$-homomorphism induced by $t$.
\end{corollary}

\section{Quantum Cuntz--Krieger algebras and local relations}\label{sec:QCK_and_local_relations}

In this section we recall the quantum Cuntz--Krieger relations and define local quantum Cuntz--Krieger relations. The former differ slightly from those appearing in \cite[Section 3.2]{BEVW20} (see Remark~\ref{rem:unital_relation}). We will see in Theorem~\ref{thm:universal_local_quantum_Cuntz--Krieger_algebra} below that the Cuntz--Pimsner algebra for a faithful quantum edge correspondence plays the role of the universal C*-algebra generated by local quantum Cuntz--Krieger relations. This in turn allows us to deduce that such Cuntz--Pimsner algebras are quotients of quantum Cuntz--Krieger algebras (see Corollary~\ref{cor:quantum_Cuntz--Krieger_quotients}).

\begin{definition}\label{def:QCK}
Let $\cG=(B,\psi,A)$ be a directed quantum graph. We define a \textit{quantum Cuntz--Krieger $\cG$-family} in a unital C*-algebra $D$ to be a linear map $s\colon B\to D$ such that:
    \begin{enumerate}[label=(\roman*)]
        \item $\mu_D(\mu_D\otimes 1)(s\otimes s^* \otimes s)(m^* \otimes 1)m^*=s$ $\hfill (\textbf{QCK1})$
    
        \item $\mu_D(s^*\otimes s)m^* = \mu_D(s\otimes s^*)m^*A$ $\hfill (\textbf{QCK2})$
        
        \item $\mu_D (s \otimes s^*)m^*(1_B) = \frac{1}{\delta^2} 1_D$ $\hfill (\textbf{QCK3})$
    \end{enumerate}
where $\mu_D\colon D\otimes D\to D$ is the multiplication map for $D$ and  $s^*(b)=s(b^*)^*$ for $b\in B$. Then the \textit{quantum Cuntz--Krieger algebra} associated to $\cG$ is the universal unital C*-algebra  $\mathcal{O}(\cG)$ generated by the image of a quantum Cuntz--Krieger $\cG$-family $S\colon B\to \cO(\cG)$.
\end{definition} 

We show that Definition~\ref{def:QCK} gives the classical Cuntz--Krieger algebra when $\cG$ is a classical graph.

\begin{example}
Let $\cG=(B,\psi, A)$ be a classical graph, i.e., $B=C(V)$ is the finite-dimensional commutative C*-algebra arising from a simple finite directed graph $G=(V, E)$ on $N = |V|$ vertices, $\psi$ is the normalized trace on $B$, and $A$ is defined by the adjacency matrix $A_G$ on $G$ in the usual way. If $\left\{e_1, \ldots, e_N\right\}$ is the canonical basis of minimal projections in $B$, it is easy to see that $\psi(e_i)=1/N$, $m(e_i\otimes e_j)=\delta_{i=j}e_j$, and $m^*(e_i)=Ne_i\otimes e_i$ for all $i, j$. Moreover, $\psi$ is a $\delta$-form with $\delta^2=N$.

Let $S:B\to \cO(\cG)$ be a universal quantum Cuntz--Krieger $\cG$-family and define $S_i=NS(e_i) \in \cO(\cG)$. It was shown in \cite[Proposition 4.1]{BEVW20} that the $S_i$ are partial isometries satisfying $S_i^*S_i=\sum_{j=1}^NA_G(i,j)S_jS_j^*$ since $S:B\to \cO(\cG)$ satisfies (\textbf{QCK1}) and (\textbf{QCK2}).  One can see that the $S_i$ have mutually orthogonal range projections by observing that they sum to the identity in $\cO({\cG})$ since $S$ satisfies (\textbf{QCK3}). Explicitly, observe that
\begin{align*}
    \sum_{i=1}^N S_iS_i^*&=N^2\left[\sum_{i=1}^N S(e_i)S^*(e_i)\right]\\
    &=N\left[\sum_{i=1}^n \mu_{\cO(\cG)}(S\otimes S^*)m^*(e_i)\right]\\
    &= N\left[\frac{1}{\delta^2}1_{\cO(\cG)}\right]
    =1_{\cO(\cG)}.
\end{align*}

Thus, the $S_i$ form a Cuntz--Krieger $A_G$-family, which induces a $*$-homomorphism of $\cO_{A_E}$ onto $\cO(\cG)$. 

Conversely, given a universal Cuntz--Krieger $A_G$ family $\left\{S_i\right\}$, we can define $s:B \to \cO_{A_G}$ via $s(e_i)=\frac{1}{N}S_i\in \cO_{A_G}$. As mentioned in \cite[Proposition 4.1]{BEVW20}, one can check that $s$ satisfies (\textbf{QCK1}) and (\textbf{QCK2}). To see that $s$ satisfies (\textbf{QCK3}), consider
\begin{align*}
    \mu_{\cO_{A_G}}(s\otimes s^*)m^*(1_B) &= \sum_{i=1}^N\mu_{\cO_{A_G}}(s(Ne_i)\otimes s^*(e_i))\\
    &=\frac{1}{N}\sum_{i=1}^NS_iS_i^*\\
    &=\frac{1}{N}1_{\cO_{A_G}}.
\end{align*}
Hence, $s$ is a quantum Cuntz--Krieger $\cG$-family, which induces a $*$-homomorphism of $\cO(\cG)$ onto $\cO_{A_G}$. Checking that this map is the inverse of the previously induced $*$-homomorphism of $\cO_{A_G}$ onto $\cO(\cG)$ yields $\cO(\cG)$ is isomorphic to $\cO_{A_G}.$
\end{example}

\begin{remark}\label{rem:unital_relation}
In \cite{BEVW20}, a notion of quantum Cuntz--Krieger algebras was introduced without the relation $(\textbf{QCK3})$ (see \cite[Definition 3.7]{BEVW20}), which gives potentially non-unital, non-nuclear C*-algebras denoted $\mathbb{F}\cO(\cG)$. As discussed in \cite[Section 4.1]{BEVW20}, when $\cG$ is a classical graph $\mathbb{F}\cO(\cG)$ is a \emph{free} Cuntz--Krieger algebra (see \cite[Definition 2.5]{BEVW20}), whereas by the above example $\cO(\cG)$ is a Cuntz--Krieger algebra. Thus we will generally refer to $\mathbb{F}\cO(\cG)$ as the \emph{free} quantum Cuntz--Krieger algebra, and reserve the terminology ``quantum Cuntz--Krieger algebra'' for $\cO(\cG)$.
\end{remark}

Using adapted matrix units, one can produce a more explicit presentation of the quantum Cuntz--Krieger relations.  This is essentially the content of \cite[Proposition 3.9]{BEVW20}.  More precisely, if $s_{ij}^{(a)}:=s(f_{ij}^{(a)})$, where $\{f_{ij}^{(a)}\colon 1\leq a \leq d,\ 1\leq i,j\leq N_a\}$ are the adapted matrix units for $(B,\psi)$, then $s:B \to D$ is a quantum Cuntz--Krieger $\cG$-family if and only if the following relations hold:

     \begin{align}
         \sum_{k,\ell=1}^{N_a} s_{ik}^{(a)}(s_{\ell k}^{(a)})^*s_{\ell s}^{(a)} &=s_{is}^{(a)} \label{eqn:QCK1}\\
        \sum_{\ell =1}^{N_a}(s_{\ell i}^{(a)})^*s_{\ell j}^{(a)} &= \sum_{c=1}^d\sum_{\ell,m,n =1}^{N_c} A_{ija}^{\ell m c} s_{\ell n}^{(c)} (s_{m n}^{(c)})^*  \label{eqn:QCK2}\\
        \sum_{c=1}^d \sum_{\ell,m=1}^{N_c} \psi(e_{\ell,\ell}^{(c)}) s_{\ell m}^{(a)} (s_{\ell m}^{(c)})^* &=\frac{1}{\delta^2} 1_D  \label{eqn:QCK3}
    \end{align}

We now introduce localized versions of the above quantum Cuntz--Krieger relations.

\begin{definition}
Let $\cG=(B,\psi,A)$ be a directed quantum graph such that $\psi$ is a $\delta$-form. We define a \textit{local quantum Cuntz--Krieger $\cG$-family} in a unital C*-algebra $D$ to be a linear map $s\colon B\to D$ such that
    \begin{enumerate}[label=(\roman*)]
        \item $\mu_D(\mu_D \otimes 1)(s \otimes s^* \otimes s)(m^* \otimes 1) = \frac{1}{\delta^2}s m$ $\hfill(\textbf{LQCK1})$
        
        \item $\mu_D (s^* \otimes s) = \frac{1}{\delta^2} \mu_D(s \otimes s^*)m^*Am$ $\hfill(\textbf{LQCK2})$
        
        \item $\mu_D (s \otimes s^*)m^*(1_B) = \frac{1}{\delta^2} 1_D$ $\hfill (\textbf{LQCK3})$
        
    \end{enumerate}
where $\mu_D\colon D\otimes D\to D$ is the multiplication map for $D$, $s^*(b)=s(b^*)^*$ for $b\in B$, and $m^*$ is the adjoint of $m$ with respect to the inner product given by $\psi$.
\end{definition}
Once again, using adapted matrix units, one can produce a more explicit presentation of the local quantum Cuntz--Krieger relations. This is the content of the following proposition whose proof is omitted as it is very similar to the proof of \cite[Proposition 3.9]{BEVW20}.

\begin{proposition}\label{prop:explicit_local_relations}
Let $\cG=(B,\psi,A)$ be a directed quantum graph such that $\psi$ is a $\delta$-form, and let $s\colon B\to D$ be a linear map into a unital C*-algebra. Denote $s_{ij}^{(a)}:=s(f_{ij}^{(a)})$, where $\{f_{ij}^{(a)}\colon 1\leq a \leq d,\ 1\leq i,j\leq N_a\}$ are the adapted matrix units for $(B,\psi)$. Then $s$ is a local quantum Cuntz--Krieger $\cG$-family if and only if the following relations hold:
    \begin{align}
         \sum_{k=1}^{N_a} s_{ik}^{(a)}(s_{jk}^{(a)})^*s_{rs}^{(b)} &=\delta_{\substack{a=b \\ j=r}} \frac{1}{\delta^2\psi(e_{jj}^{(a)})}s_{is}^{(a)} \label{eqn:QCP1}\\
        (s_{ij}^{(a)})^*s_{rs}^{(b)} &= \delta_{\substack{a=b\\ i=r}}\frac{1}{\delta^2\psi(e_{ii}^{(a)})}\sum_{c=1}^d\sum_{\ell,m,n =1}^{N_c} A_{jsa}^{\ell m c} s_{\ell n}^{(c)} (s_{m n}^{(c)})^*  \label{eqn:QCP2}\\
        \sum_{c=1}^d \sum_{\ell,m=1}^{N_c} \psi(e_{\ell,\ell}^{(c)}) s_{\ell m}^{(a)} (s_{\ell m}^{(c)})^* &=\frac{1}{\delta^2} 1_D  \label{eqn:QCP3}
    \end{align}
for all $1\leq a,b\leq d$, $1\leq i,j\leq N_a$, and $1\leq r,s\leq N_b$.
\end{proposition}

Just as in the case of the quantum Cuntz--Krieger algebras above, one can also define, for any quantum graph $\cG = (B,\psi, A)$, a corresponding {\it local quantum Cuntz--Krieger algebra}. The local quantum Cuntz--Krieger algebra is the C$^\ast$-algebra generated by a universal local quantum Cuntz--Krieger $\cG$-family.
The following theorem is the main result of this section.  It shows that local quantum Cuntz--Krieger algebras are in fact  familiar objects -- under rather mild assumptions, they are precisely the Cuntz--Pimsner algebras of quantum edge correspondences.

\begin{theorem}\label{thm:universal_local_quantum_Cuntz--Krieger_algebra}
Let $\cG=(B,\psi,A)$ be a directed quantum graph such that $\psi$ is a $\delta$-form and $A$ is completely positive, and let $E_\cG=B\cdot \epsilon_\cG\cdot B$ be its quantum edge correspondence. Let $(\pi_{E_\cG},t_{E_\cG})$ denote the universal covariant representation of $E_\cG$ on the Cuntz--Pimsner algebra $\cO_{E_\cG}$.
Assume $\cG$ has no quantum sources. 
Then $S\colon B\to \mathcal{O}_{E_\cG}$ defined by $S(x):=\frac{1}{\delta} t_{E_{\cG}}(x\cdot \epsilon_\cG)$ is a local quantum Cuntz--Krieger $\cG$-family whose image generates $\cO_{E_\cG}$. Moreover, given any local quantum Cuntz--Krieger $\cG$-family $s\colon B\to D$ in a unital C*-algebra $D$ there exists a $*$-homomorphism $\rho\colon \mathcal{O}_{E_\cG}\to D$ such that $s=\rho\circ S$.
\end{theorem}
\begin{proof}
We begin by showing that $S$ is a local quantum Cuntz--Krieger $\cG$-family. Let $\psi_{t_{E_\cG}} \colon \cK(E_\cG)\to \cO_{E_\cG}$ be the $*$-homomorphism induced by $t_{E_\cG}$, and let $\mu\colon \cO_{E_\cG}\otimes \cO_{E_\cG}\to \cO_{E_\cG}$ denote the multiplication map. Define a linear map $T\colon B\to \cO_{E_\cG}$ by $T(x):=t_{E_\cG}(x\cdot \epsilon_\cG)$, so that $S=\frac{1}{\delta} T$. Note that the quantum edge correspondence $E_\cG$ is faithful by Theorem~\ref{thm:faithful_full_correspondence} and finite dimensional so that the Katsura ideal $J_{E_\cG}=B$. Moreover, faithfulness allows us to identify $B\subset \cL(E_\cG)=\cK(E_\cG)$, and under this identification we have $\psi_{t_{E_\cG}}|_B = \pi$. With this observation in hand, we can readily verify the local quantum Cuntz--Krieger relations.

    \begin{itemize}
        \item[]\textbf{(LQCK1):} The second equation in Corollary~\ref{cor:abstract_Toeplitz} implies
            \begin{align*}
                \mu(\mu\otimes 1)(S\otimes S^*\otimes S)(m^*\otimes 1) &= \frac{1}{\delta^3} \mu(\mu\otimes 1)(T\otimes T^*\otimes T)(m^*\otimes 1)\\
                &= \frac{1}{\delta^3} \mu(\psi_{t_{E_\cG}}\otimes T)\\
                &=\frac{1}{\delta^3}\mu(\pi\otimes T)\\
                &= \frac{1}{\delta^3} Tm= \frac{1}{\delta^2} Sm,
            \end{align*}
        where the second-to-last equality follows from the relation $\pi_{E_\cG}(x)t_{E_\cG}(\xi) = t_{E_\cG}(x\cdot \xi)$ for $x\in B$ and $\xi\in E_\cG$.
        
        \item[]\textbf{(LQCK2):} Using both equations in Corollary~\ref{cor:abstract_Toeplitz} gives
            \begin{align*}
                \mu(S^*\otimes S) &= \frac{1}{\delta^2} \mu(T^*\otimes T) = \frac{1}{\delta^4} \pi A m = \frac{1}{\delta^2} \psi_{t_{E_\cG}} A m\\
                    &= \frac{1}{\delta^4} \mu(T\otimes T^*)m^* Am = \frac{1}{\delta^2} \mu(S\otimes S^*)m^*Am.
            \end{align*}
            
        \item[]\textbf{(LQCK3):} Using the second equation in Corollary~\ref{cor:abstract_Toeplitz} we have
            \[
                \mu(S\otimes S^*)m^*(1_B) = \frac{1}{\delta^2} \mu(T\otimes T^*)m^*(1_B) = \frac{1}{\delta^2} \pi_{E_{\cG}}(1_B) = \frac{1}{\delta^2} 1.
            \]
        Note that $\pi_{E_{\cG}}$ is necessarily unital since $1\cdot \xi=\xi$ for all $\xi\in E_{\cG}$.
    \end{itemize}
Thus $S$ is a local quantum Cuntz--Krieger $\cG$-family. We also have $C^*(S(B))=\cO_{E_\cG}$. Indeed, $\pi_{E_\cG}(B)\subset C^*(S(B))$ by Corollary~\ref{cor:abstract_Toeplitz}, and $t_{E_\cG}(\xi)\pi_{E_\cG}(b)=t_{E_\cG}(\xi\cdot b)$ implies $t_{E_\cG}(E_\cG)\subset C^*(S(B))$.

Now, suppose $s\colon B\to D$ is a local quantum Cuntz--Krieger $\cG$-family in a unital C*-algebra $D$. To obtain the desired homomorphism, we will construct a covariant representation of $E_\cG$ on $D$ and invoke the universal property of the Cuntz--Pimsner algebra $\cO_{E_\cG}$. Define a linear map $\pi\colon B\to D$ by
    \[
        \pi := \delta^2 \mu_D(s\otimes s^*)m^*,
    \]
where $\mu_D\colon D\otimes D\to D$ is the multiplication map. Then $\pi$ is unital by (\textbf{LQCK3}), $\pi^*=\pi$ in light of how the multiplication maps interact with the adjoint, and (\textbf{LQCK1}) implies
    \begin{align*}
        \mu_D(\pi\otimes \pi) &= \delta^4 \mu_D(\mu_D\otimes \mu_D)(s\otimes s^*\otimes s\otimes s^*)(m^*\otimes m^*)\\
            &= \delta^4 \mu_D( \left[\mu_D(\mu_D\otimes 1)(s\otimes s^*\otimes s)(m^*\otimes 1)\right]\otimes s^*)(1\otimes m^*)\\
            &= \delta^2 \mu_D( (sm)\otimes s^*)(1\otimes m^*)\\
            &= \delta^2 \mu_D( s\otimes s^*)m^*m = \pi m.
    \end{align*}
Hence $\pi$ is a $*$-homomorphism. 

Next define a linear map $t\colon E_\cG\to D$ by
    \[
        t(x\cdot \epsilon_\cG\cdot y):=\delta s(x)\pi(y) \qquad \qquad x,y\in B.
    \]
We will see below that $t(\xi)^*t(\eta)=\pi(\<\xi,\eta\>_B)$ for $\xi,\eta\in B$, which in particular will show this map is well-defined. By (\textbf{LQCK1}) we have
    \begin{align*}
        \mu_D(\mu_D\otimes 1)(\pi\otimes s\otimes \pi) &= \delta^2 \mu_D(\mu_D\otimes 1)(\mu_D\otimes 1\otimes 1)(s\otimes s^*\otimes s\otimes \pi)(m^*\otimes 1\otimes 1)\\
            &=  \mu_D( s\otimes \pi)(m\otimes 1).
    \end{align*}
Thus for $x,y,z\in B$ we have
    \begin{align*}
        \pi(x)t(y\cdot \epsilon_\cG\cdot z) &= \delta \mu_D(\mu_D\otimes 1)(\pi\otimes s\otimes \pi)(x\otimes y\otimes z)\\
            &= \delta \mu_D(s\otimes \pi)(m\otimes 1)(x\otimes y \otimes z) = s(xy)\pi(z) = t(xy\cdot \epsilon_\cG\cdot z).
    \end{align*}
and so $\pi(x)t(\xi)=t(x\cdot \xi)$ for all $x\in B$ and $\xi\in E_\cG$. Using (\textbf{LQCK2}) and the definition of $\pi$, we have
    \[
        \mu_D(s^*\otimes s) = \frac{1}{\delta^2} \mu_D(s\otimes s^*)m^* Am = \frac{1}{\delta^4} \pi A m.
    \]
Thus for $x,x',y,y'\in B$ we have
    \begin{align*}
        t(x\cdot \epsilon_\cG \cdot y)^* t(x'\cdot \epsilon_\cG \cdot y') &= \delta^2  (\mu_D\otimes \mu_D)(1\otimes \mu_D\otimes 1)(\pi^*\otimes s^*\otimes s\otimes \pi)(y^*\otimes x^*\otimes x'\otimes y')\\
            &=\frac{1}{\delta^2} (\mu_D\otimes \mu_D)(\pi \otimes (\pi A)\otimes \pi)(1\otimes m\otimes 1)(y^*\otimes x^*\otimes x'\otimes y')\\
            &= \frac{1}{\delta^2} \pi( y^* A(x^*x') y')\\
            &= \pi\left(y^* \<x\cdot \epsilon_\cG, x'\cdot \epsilon_\cG\>_B y'\right)  = \pi\left(\< x\cdot \epsilon_\cG\cdot y, x'\cdot \epsilon_\cG\cdot y'\>_B \right),
    \end{align*}
where the second-to-last equality follows from Theorem~\ref{thm:cp_Adjacency_matrices_are_inner_products}. Thus $(\pi,t)$ is a representation of $E_\cG$ on $D$, and Theorem~\ref{thm:compact_operators} implies it is covariant:
    \begin{align*}
        \psi_t(f_{ij}^{(a)}) &= \sum_{k=1}^{N_a} \psi_t\left( \theta_{f_{ik}^{(a)}\cdot \epsilon_\cG, f_{jk}^{(a)}\cdot \epsilon_\cG} \right)\\
            &=\sum_{k=1}^{N_a} t(f_{ik}^{(a)}\cdot \epsilon_\cG) t( f_{jk}^{(a)}\cdot \epsilon_\cG)^*\\
            &= \delta^2 \mu_D(s\otimes s^*)m^*(f_{ij}^{(a)}) = \pi(f_{ij}^{(a)}).
    \end{align*}
Thus the universal property for the Cuntz--Pimsner algebra $\cO_{E_\cG}$ implies there is a $*$-homomorphism $\rho\colon \cO_{E_\cG}\to D$ satisfying $\pi=\rho\circ \pi_{E_\cG}$ and $t=\rho\circ t_{E_\cG}$. In particular, we have
    \[
        s(x) = \frac{1}{\delta} t(x\cdot \epsilon_\cG) = \frac{1}{\delta} \rho\circ t_{E_\cG}(x\cdot \epsilon_\cG) = \rho(S(x)).
    \]
Hence $s=\rho\circ S$.
\end{proof}

Observe that any local quantum Cuntz--Krieger $\cG$-family is a (non-local) quantum Cuntz--Krieger $\cG$-family. Indeed, since $\psi$ is a $\delta$ form one has $mm^*=\delta^2 \id$. Thus if $s\colon B\to D$ satisfies (\textbf{LQCK1}) and (\textbf{LQCK2}), then applying $m^*$ to the right-hand sides of these relations yields (\textbf{QCK1}) and (\textbf{QCK2}), respectively. Also (\textbf{LQCK3}) and (\textbf{QCK3}) are identical. Hence the universal property for $\mathcal{O}(\cG)$ yields a unique $*$-homomorphism onto $C^*(s(D))$. In particular, if $\ker(A)$ does not contain a central summand of $B$, then the previous theorem yields the following:

\begin{corollary}\label{cor:quantum_Cuntz--Krieger_quotients}
Let $\cG=(B,\psi,A)$ be a directed quantum graph such that $\psi$ is a $\delta$-form and $A$ is completely positive, and let $E_\cG$ be its quantum edge correspondence. Assume that $\ker(A)$ does not contain a central summand of $B$. Then $\mathcal{O}_{E_\cG}\cong \mathcal{O}(\cG)/\mathcal{I}$ where $\mathcal{I}\triangleleft \mathcal{O}(\cG)$ is the closed two-sided ideal generated by the relations (\textbf{LQCK1}), (\textbf{LQCK2}), and (\textbf{LQCK3}).
\end{corollary}

\subsection{Behavior of $\cO_{E_\cG}$ under quantum graph isomorphisms}
In this section we briefly examine the relationship between the Cuntz--Pimsner algebras of quantum edge correspondences associated to quantum isomorphic quantum graphs.   We begin by recalling the notion of quantum isomorphism from \cite{BCEHPSWCMP19, BEVW20}.  Let $\cG_i=(B_i,\psi_i,A_i)$, $i=1,2$, be directed quantum graphs. We say that $\cG_1$ and $\cG_2$ are quantum isomorphic if there exists a Hilbert space $\cH$ and a unital $\ast$-homomorphism \[\theta_1:B_1 \to B_2 \otimes B(\cH)\]
which satisfies the following $\psi_i$ and $A_i$ covariance conditions
    \begin{align*}
        (\psi_2 \otimes \id)\theta_1 &=\psi_1(\cdot)1_{B(\cH)}, \\
        (A_2 \otimes \id)\theta_1 &= \theta_1 \circ A_1.
    \end{align*}
Note that the general theory guarantees that whenever a morphism $\theta_1:B_1 \to B_2 \otimes B(\cH)$ exists, there automatically exists a corresponding morphism $\theta_2:B_2 \to B_1 \otimes B(\cH)$ with analogous covariance conditions to those for $\theta_1$.  In particular, the notion of quantum isomorphism is symmetric in $\cG_1$ and $\cG_2$. 

Quantum isomorphisms between quantum graphs can be seen as relaxations (or generalizations) of the notion of an isomorphism between quantum graphs.  Indeed, in the special case where $\cH = \bC$, a quantum isomorphism between $\cG_1$ and $\cG_2$ defines  an ordinary isomorphism of quantum graphs. 

The following theorem should be compared with \cite[Theorem 6.13]{BEVW20}, which is a similar result in the context of (free) quantum Cuntz--Krieger algebas.  Note, however, that the conclusion of the following theorem is slightly stronger in that one gets {\it injective} morphisms $\Theta_i$ between Cuntz--Pimsner algebras below (compare with \cite[Remark 6.14]{BEVW20}).     

\begin{theorem}
Let $\cG_i=(B_i,\psi_i,A_i)$, $i=1,2$, be directed quantum graphs.  If $\cG_1$ is quantum isomorphic to $\cG_2$ with $*$-homomorphisms
    \[
        \theta_1\colon B_1\to B_2\otimes B(\cH)\qquad \text{and} \qquad \theta_2\colon B_2\to B_1\otimes B(\cH),
    \]
    as above, then there exist injective $*$-homomorphisms
    \[
        \Theta_1\colon \cO_{E_{\cG_1}}\to \cO_{E_{\cG_2}}\otimes B(\cH) \qquad \text{and} \qquad  \Theta_2\colon \cO_{E_{\cG_2}}\to \cO_{E_{\cG_1}}\otimes B(\cH)
    \]
satisfying
    \[
        \Theta_1\pi_{E_{\cG_1}} = (\pi_{E_{\cG_2}}\otimes 1)\theta_1 \qquad \text{and}\qquad \Theta_2\pi_{E_{\cG_2}} =  (\pi_{E_{\cG_1}}\otimes 1)\theta_2.
    \]
\end{theorem}
\begin{proof}
Using Corollary~\ref{cor:quantum_graph_correspondence_as_cp_correspondence}, we can work with $B_i\otimes_{A_i} B_i$ rather than $E_{\cG_i}$, $i=1,2$. Denote $\tilde{A}_2:=A_2\otimes \text{id}_{B(\cH)}$, which is a completely positive map, and $X:=(B_2\otimes B(\cH))\otimes_{\tilde{A}_2} (B_2\otimes B(\cH))$. Note that $X\cong (B_2\otimes_{A_2} B_2)\otimes B(\cH)$, and so $\cO_X \cong \cO_{B_2\otimes_{A_2} B_2}\otimes B(\cH)$ by \cite[Example 6.4]{Mor17}. 

Now, for $x,y\in B_1$ and $\xi\in B_1\otimes B_1$ we have
    \[
        (\tilde{A}_2\otimes 1)( (\theta_1\otimes \theta_1)(\xi)^* (\theta_1\otimes \theta_1)(\xi) ) = (\theta_1\otimes \theta_1)(A_1\otimes 1)(\xi^*\xi).
    \]
It follows that $\theta_1\otimes \theta_1$ induces a linear map $T\colon B_1\otimes_{A_1} B_1\to X$ satisfying
    \[
        T(x\cdot \xi \cdot y) = \theta_1(x)\cdot T(\xi)\cdot \theta_1(y) \qquad x,y\in B_1,\ \xi\in B_1\otimes_{A_1} B_1,
    \]
and
    \[
        \<T(\xi),T(\eta)\>_{B_2\otimes B(\cH)} = \theta_1(\<\xi, \eta\>_{B_1}) \qquad \xi,\eta\in B_1\otimes_{A_1} B_1.
    \]
Thus if $(\pi_X,t_X)$ is the universal covariant representation of $X$ (note that we can take $\pi_X:= \pi_{E_{\cG_2}}\otimes 1$ and $t_X:=t_{E_{\cG_2}}\otimes 1$), then $(\pi_X\circ \theta_1, t_X\circ T)$ is a covariant representation of $B_1\otimes_{A_1} B_1$ on $\cO_X$. One easily checks that this representation is injective and admits a gauge action, and so \cite[Theorem 6.4]{Kat04} implies there is an injective $*$-homomomorphism $\Theta_1\colon \cO_{B_1\otimes_{A_1} B_1}\to \cO_X$. Reversing the roles of $B_1$ and $B_2$ yields $\Theta_2$.
\end{proof}

\section{Examples}\label{sec:simpler_examples}

In this section we consider three common types of quantum graphs and determine the isomorphism classes of the Cuntz--Pimsner algebras associated to their quantum edge correspondences. Notably, several of these examples can be realized as Exel crossed products associated to natural Exel systems. We refer the reader to \cite{Exel03-2}, \cite{Exel03}, and \cite{Exel04} for details but provide a brief summary of the construction of Exel crossed products below.

Let $\cC$ be a C*-algebra and $\alpha:\cC \to \cC$ be $\ast$-endomorphism. A {\em transfer operator} for $(\cC,\alpha)$ is a positive linear map on $\cC$ such that $\mathcal{L}(\alpha(f)g)=f\mathcal{L}(g)$ for all $f,g\in \cC$. The Exel crossed product is obtained by first constructing a Toeplitz algebra $\cT(\cC,\alpha,\mathcal{L})$ which is the universl C$^\ast$-algebra generated by a copy of $\cC$ along with an element $S$ such that $Sf=\alpha(f)S$ and $\mathcal{L}(f)=S^*fS$ for all $ f\in \cC.$ Two elements $f,k\in \cT(\cC,\alpha,\mathcal{L})$ form a {\em redundancy} if $f\in \cC, k\in \overline{\cC SS^*\cC},$ and $fgS=kgS$ for all $g\in \cC.$ The {\em Exel crossed product} $\cC\rtimes_{\alpha,\mathcal{L}} \mathbb{N}$ is the quotient of $\mathcal{T}(\cC,\alpha,\mathcal{L})$ by the ideal generated by differences $f-k$ of redundancies $(f,k)$. 

Given an $n\times n$ $\{0,1\}$-matrix $A$, one can construct the Markov subshift space of infinite paths in the graph associated to $A$, denoted by $X_A$.  $X_A$ is a compact space, and admits a natural shift action $\sigma:X_A \to X_A$ given by $\sigma (x_1,x_2,x_3, \ldots ) = (x_2,x_3, \ldots)$.  A nice class of commutative Exel systems $(C(X_A),\alpha,\cL)$ arises from this construction, where $\alpha(f) = f \circ \sigma$ for $f \in C(X_A)$ and \[\mathcal L(f) (x) = |\{y \in X_A \ | \ \sigma(y) = x\}|^{-1} \Big(
\sum_{\{y \in X_A \ | \ \sigma(y) = x\}}f(y)\Big).\]  It is known that the Cuntz--Krieger algebra $\cO_A$ is isomorphic to the Exel crossed product $ C(X_A)\rtimes_{\alpha,\cL} \bN$. Below, we explore the analogous relationship between noncommutative Exel systems and local quantum Cuntz--Krieger algebras.

\subsection{Complete quantum graphs}\label{subsection:completequantumgraphs}
Given a finite quantum space $(B,\psi)$ with $\delta$-form $\psi$, we denote by $K(B,\psi)$ the {\em complete quantum graph} $(B,\psi, A)$, whose quantum adjacency matrix $A$ is given by the rank-one map  $A=\delta^2\psi(\cdot)1_B$ \cite{BCEHPSWCMP19, BEVW20}.  The quantum edge correspondence for $K(B,\psi)$ is $E_K=B\otimes B$ since an adapted matrix unit computation yields that the quantum edge indicator is $\epsilon_K=1_B\otimes 1_B.$ The quantum Cuntz--Krieger algebra $ \cO(K(B,\psi))$ is generated by $\{
S_{ij}^{(a)}: 1 \le a \le d, \ 1 \le i,j \le N_a\}$
according to the relations
\begin{align} 
\label{QK1}
&(S^{(a)})^*S^{(a)} 
= 
1 \quad \text{in } M_{N_a}\otimes \mathcal{O}(K(B,\psi)) \qquad (1 \le a \le d)\\
\label{QK2}
& \sum\limits_{ars} \delta^2 \psi(e_{rr}^{(a)})S_{rs}^{(a)}S_{rs}^{(a)*}
=
1 \quad \text{in } \mathcal{O}(K(B,\psi))
\end{align}

In \cite[Theorem 4.5]{BEVW20}, the authors establish that their definition of a quantum Cuntz--Krieger algebra associated to $K(B,\psi)$ is isomorphic to $\cO_{\dim B}$ whenever $\delta^2 \in \mathbb N$. It is unknown if this also holds for $\delta^2\not\in \mathbb{N}$. 
The following proposition resolves this for local quantum Cuntz--Krieger algebras.

\begin{proposition} \label{prop:Cuntz--Pimsner-is-a-Cuntz-algebra}
$\cO_{E_K}$ is isomorphic to the Cuntz algebra $\cO_{n}$ where $n = \dim B$. 
\begin{proof}
By Theorem \ref{thm:universal_local_quantum_Cuntz--Krieger_algebra} and Proposition \ref{prop:explicit_local_relations}, the Cuntz--Pimsner algebra $\mathcal{O}_{E_{K}}$ is a quotient of the quantum Cuntz--Krieger algebra $\mathcal{O}(K(B,\psi))$ subject to relations (\ref{eqn:QCP1}), which remains the same, and (\ref{eqn:QCP2}), which in this case becomes: 

\begin{align}
    (S_{ij}^{(a)})^*(S_{rs}^{(b)})
    &=\delta_{\substack{a=b\\ i=r}}
    \frac{1}{\delta^2 \psi(e_{ii}^{(a)})}
    \sum\limits_{c\ell m}
    \left[
    \delta^2 \delta_{\substack{\ell=m\\ j=s}}\psi(e_{\ell \ell}^{(c)})
    \sum\limits_n S_{\ell n}^{(c)}(S_{mn}^{(c)})^*
    \right].\label{eq:2}
\end{align}
Simplifying equation \eqref{eq:2} using the unit relation (\ref{eqn:QCP3}), we get

\[
(S_{ij}^{(a)})^*(S_{rs}^{(b)})=\frac{\delta_{\substack{a=b\\ i=r\\ j=s}}}{\delta^2 \psi(e_{ii}^{(a)})}\left(\sum\limits_{c\ell n}\delta^2\psi(e_{\ell \ell}^{(c)}) S_{\ell n}^{(c)}(S_{\ell n}^{(c)})^*\right) = \frac{\delta_{\substack{a=b\\ i=r\\ j=s}}}{\delta^2 \psi(e_{ii}^{(a)})}(1_{\mathcal{O}(E_{K})}).
\]
Thus, $\left\{\delta\psi(e_{ii}^{(a)})^{1/2}S_{ij}^{(a)}:1\leq a\leq d, 1\leq i,j\leq N_a\right\}$ are generating isometries for $\mathcal{O}_{E_{K}}$ whose range projections sum to the identity, and it follows that $\mathcal{O}_{E_{K}}$ is isomorphic to the Cuntz algebra $\mathcal{O}_{\dim B}.$
\end{proof}
\end{proposition}

Here, we define a particular noncommutative Exel system $(B^{\otimes \bN}, \alpha, \mathcal{L})$ and show that the Exel crossed product $B^{\otimes \bN}\rtimes _{\alpha,\mathcal{L}}\bN$ is the Cuntz algebra on $\dim B$ generators. Define $\alpha:B^{\otimes \bN}\to B^{\otimes \bN}$ by \[\alpha(f):=1_B\otimes f \qquad (f\in B^{\otimes \bN}).\] Then $\mathcal{L}:B^{\otimes \bN} \to B^{\otimes \bN}$ given by \[\mathcal{L}(f_1\otimes g):=\psi(f_1)g, \qquad (f_1 \in B, g \in B^{\otimes \bN}),\] is a transfer operator for the dynamical system $(B^{\otimes \bN},\alpha)$. In the rest of this section, we let $1$ denote the identity in $B^{\otimes \bN}$ and continue to denote the identity in $B$ by $1_B$. For each $1\leq a \leq d$ and $1\leq i,j \leq N_a$, define $E_{ij}^{(a)}:=e_{ij}^{(a)}\otimes 1$, and let $S$ denote the operator which arises in the construction of $\cT(B^{\otimes \bN},\alpha,\mathcal{L})$. We will show that this particular Exel crossed product is $\mathcal{O}_{\dim B}$, generated by $\{v_{ij}^{(a)}:1\leq a\leq d, 1\leq i,j\leq N_a\}$, where $$v_{ij}^{(a)}:=\frac{1}{\psi(e_{jj}^{(a)})^{-1/2}}E_{ij}^{(a)}S.$$ 
The $v_{ij}^{(a)}$ should be regarded as an analogue of the characteristic function of words in $X_A$ that ``begin with the quantum letter $e_{ij}^{(a)} \in B$.''

\begin{lemma}\label{lem:Complete_graph_Exel_redundancies}
For all $1\leq a\leq d$, $1\leq i,j \leq N_a$, the pair $\left(E_{ij}^{(a)}, \sum_{p} v_{ip}^{(a)}{v_{jp}^{(a)}}^*\right)$ is a redundancy.
\end{lemma}
\begin{proof}
Fix $a\in \{1,...,d\}$ and $i,j\in \{1,...,N_a\}$. For $e_{k\ell}^{(b)}\otimes f\in B^{\otimes \bN},$ note $v_{jp}^{(a)*}(e_{k\ell}^{(b)}\otimes f)=0$ unless $a=b$ and $j=k.$ Thus, it suffices to consider elements of the form $e_{j\ell}^{(a)}\otimes f$ in $B^{\otimes \bN}$. Observe:
\begin{align*}
    \psi(e_{pp}^{(a)})(v_{ip}^{(a)}{v_{jp}^{(a)*}})(e_{j\ell}^{(a)}\otimes f)S
    &=E_{ip}^{(a)}SS^*[E_{pj}^{(a)}(e_{j\ell}^{(a)}\otimes f)]S\\
    &=
    E_{ip}^{(a)}S[S^*(e_{p\ell}^{(a)}\otimes f)S]\\
    &=
    E_{ip}^{(a)}S\mathcal{L}(e_{p\ell}^{(a)}\otimes f)\\
    &=
    \psi(e_{p\ell}^{(a)})E_{ip}^{(a)}Sf\\
    &=
    \psi(e_{p\ell}^{(a)})E_{ip}^{(a)}\alpha(f)S\\
    &=
    \psi(e_{p\ell}^{(a)})E_{ip}^{(a)}(1_B\otimes f)S\\
    &=
    \psi(e_{pp}^{(a)})\delta_{p=\ell}(e_{i\ell}^{(a)}\otimes f)S\\
    &=
    \psi(e_{pp}^{(a)})\delta_{p=\ell}(e_{ij}^{(a)}e_{j\ell}^{(a)}\otimes f)S\\
    &= 
    \psi(e_{pp}^{(a)})\delta_{p=\ell}E_{ij}^{(a)}(e_{j\ell}^{(a)}\otimes f)S
\end{align*}
As elements of the form $e_{k\ell}^{(b)}\otimes f$ span $B^{\otimes \bN}$, for all $g\in B^{\otimes \bN},$ we have 
$$
\left(\sum_{p} v_{ip}^{(a)}{v_{jp}^{(a)}}^*\right)gS
=\sum_{p} \left(v_{ip}^{(a)}{v_{jp}^{(a)}}^*gS\right)
=\sum_p \left(\delta_{p\ell} E_{ij}^{(a)}gS\right)
=E_{ij}^{(a)}gS.
$$
By definition, $\left(E_{ij}^{(a)},\sum_{p} v_{ip}^{(a)}{v_{jp}^{(a)}}^*\right)$ is a redundancy.
\end{proof}

\begin{theorem}\label{thm:crossed-product-is-a-Cuntz-algebra}
The Exel crossed product $B^{\otimes \bN}\rtimes_{\alpha,\mathcal{L}}\mathbb{N}$ is isomorphic to the Cuntz algebra $\cO_n$ where $n=\dim B$.
\end{theorem}
\begin{proof}
We first show that the C*-algebra generated by $\{v_{ij}^{(a)}:1\leq a\leq d, 1\leq i,j\leq N_a\}$ is $\mathcal{O}_n,$ where $n=\dim B.$ Observe that each $v_{ij}^{(a)}$ is an isometry in $B^{\otimes \bN}\rtimes_{\alpha,\mathcal{L}}\mathbb{N}$:
$$
    {v_{ij}^{(a)}}^*v_{ij}^{(a)}
    =\frac{1}{\psi(e_{jj}^{(a)})}S^*E_{ji}^{(a)}E_{ij}^{(a)}S
    =\frac{1}{\psi(e_{jj}^{(a)})}S^*E_{jj}^{(a)}S
    =\frac{1}{\psi(e_{jj}^{(a)})}\mathcal{L}\left(E_{jj}^{(a)}\right)
    =1.
$$
By Lemma~\ref{lem:Complete_graph_Exel_redundancies}, $\sum_p v_{ip}^{(a)}{v_{ip}^{(a)}}^*=E_{ii}^{(a)}$, and thus, summing over $1\leq a\leq d$ and $1\leq i\leq N_a$ yields the identity. Therefore, each $v_{ij}^{(a)}$ is a Cuntz isometry, and the collection $\{v_{ij}^{(a)}:1\leq a\leq d, 1\leq i,j\leq N_a\}$ generates the Cuntz algebra $\cO_n$ where $n=\dim B$.

To verify that this copy of $\cO_{\dim B}$ inside $B^{\otimes \bN}\rtimes_{\alpha,\cL}\bN$ is actually all of $B^{\otimes \bN}\rtimes_{\alpha,\cL}\bN$, note that each $E_{ij}^{(a)}$ belongs to $\cO_{\dim B}$ by Lemma~\ref{lem:Complete_graph_Exel_redundancies}. Further, since $S=1_BS=(\sum_{a,i} E_{ii}^{(a)})S$, we also know that $S$ belongs to this copy of $\cO_{\dim B}$. We claim that elements of the form $E_{ij}^{(a)}$ along with $S$ generate all of $B^{\otimes \bN}$. Recall from \cite{Exel03} that $(\alpha(f),SfS^*)$ is a redundancy for all $f\in B^{\otimes \bN}$, so for any $r\in \bN$, we have
$$
\underbrace{1_B\otimes ...\otimes 1_B}_{r} \otimes \, e_{ij}^{(a)} \otimes 1 
= \alpha^r(E_{ij}^{(a)})
=S^rE_{ij}^{(a)}(S^*)^r
\in \cO_{\dim B}.
$$
As elements of the above form generate the C*-algebra $B^{\otimes \bN}$ and are contained in $\cO_{\dim B}$, we may conclude $B^{\otimes \bN}$ is a subset of $\mathcal{O}_{\dim{B}}.$ Therefore, $B^{\otimes \bN}\rtimes_{\alpha,\cL}\bN \cong \cO_n$ where $n=\dim B.$
\end{proof}

As a result of Theorem~\ref{thm:crossed-product-is-a-Cuntz-algebra}, the Cuntz--Pimsner algebra $\cO_{E_K}$ for the quantum edge correspondence of the complete quantum graph can be realized as the Exel crossed product $B^{\otimes \bN}\rtimes_{\alpha,\cL}\bN$, whose underlying Exel system was a natural choice for the complete quantum graph. We may also conclude that the Cuntz algebra on $\dim B$ generators is always a quotient of the quantum Cuntz-Krieger algebra $\cO(K(B,\psi))$ for the complete quantum graph.


\subsection{Trivial quantum graphs}
Denote by $T(B,\psi)$ the {\it trivial quantum graph} $(B,\psi, \id) $; in the following discussion $B$ and $\psi$ will be fixed, so we will simply denote it by $T$. Using the adapted matrix units, we will now identify the edge correspondence and prove that it is isomorphic to the trivial $B$-correspondence, which will allow us to explicitly compute its Cuntz--Pimsner algebra. Recall the definition of the edge correspondence from \Cref{def:edgecorrespondence}. It is generated, as a $B$-bimodule by the idempotent $\epsilon_{T}:= \frac{1}{\delta^2} (\id \otimes \id)m^{\ast}(1)=\frac{1}{\delta^2} m^{\ast}(1)$. As $m^{\ast}$ is a $B$-bimodule map, we get $x\cdot \epsilon_{T}\cdot y = \frac{1}{\delta^2} m^{\ast}(xy)$, hence the edge correspondence is equal as a vector space to the image of $m^{\ast}$. We will actually show that the map $\frac{1}{\delta}m^{\ast}$ identifies $B$ and the edge correspondence $E_{T}$ as $C^{\ast}$-correspondences.
\begin{proposition}
Let $B$ be viewed as a $C^{\ast}$-correspondence, where $\langle a, b\rangle_B =a^{\ast}b$ and the left and right actions are the usual ones. Then $\frac{1}{\delta} m^{\ast}: B \to E_{T}$ is an isomorphism of $C^{\ast}$-correspondences.
\end{proposition}
\begin{proof}
We will use \Cref{prop:recognition}. The vector $\xi:=\frac{1}{\delta}1$ is cyclic for $B$ and $\delta^2 \langle \xi, x\cdot \xi\rangle = x$ is the quantum adjacency matrix of the trivial quantum graph. Therefore by \Cref{prop:recognition} the assignment $\xi \mapsto \epsilon_{T}$ extends to an isomorphism of correspondences $B$ and $E_{T}$. This map is precisely equal to $\frac{1}{\delta}m^{\ast}$.
\end{proof}
\begin{corollary}\label{cor:triv}
The Cuntz--Pimsner algebra $\cO_{E_{T}}$ is isomorphic to $B\otimes C(\bT)$.
\end{corollary}
\begin{proof}
From the previous proposition, $E_T$ is isomorphic to $B$. By \cite[Example 3 p. 193]{Pimsner}, the Cuntz--Pimsner algebra of $B$ is isomorphic to the crossed of $B$ by the trivial action of $\mathbb{Z}$, i.e. $B\otimes C(\bT)$.
\end{proof}

Consider the trivial Exel system $(B,\alpha,\mathcal{L})$ where $\alpha=\mathcal{L}=\id$. Recall that $\mathcal{T}(B,\alpha,\mathcal{L})$ is the universal C*-algebra generated by $B$ and an element $U$ subject to the relations $Ux=\alpha(x)U$ and $U^*xU=\mathcal{L}(x)$ for all $x\in B.$ In this case, $Ux=xU$, which implies that $U$ is an isometry and commutes with every element of $B$.

\begin{proposition}
The Exel crossed product $B\rtimes_{\id,\id} \bN$ is isomorphic to $B\otimes C(\bT)$
\end{proposition}
\begin{proof}
Recall that $B\rtimes_{\id,\id} \bN$ is the quotient of $\mathcal{T}(B,\id,\id)$ by the ideal of generated by the set of redundancies, $\{a-k:a\in B,\,k\in \overline{BUU^*B},\,agU=kgU\,\,\forall g\in B\}.$ Consider $k=UU^*$ and $a=1$. Then because $U$ is isometric and commutes with all elements of $B$, we have $UU^*gU=Ug=gU$ for all $g\in B$. Therefore $UU^*$ and $1$ form a redundancy, so the quotient consists of a copy of $B$ and a unitary that commutes with $B$. Therefore, $B\rtimes_{\id, \id}\bN \cong B\otimes C(\bT)$.
\end{proof}

\begin{remark}
Note that the above proposition provides a second example of the Cuntz--Pimsner algebra for the quantum edge correspondence being implemented by a natural choice of Exel crossed product.
\end{remark}

\subsection{Rank-one quantum graphs}
In the special case when $B=M_n(\mathbb{C})$ and $\psi$ is a properly normalized trace, there is a nice correspondence between quantum adjacency matrices, projections in $M_n(\mathbb{C}) \otimes M_n(\mathbb{C})^{\op}$, and subspaces of $M_n(\mathbb{C})$. Just as before, a projection in $M_n(\mathbb{C}) \otimes M_n(\mathbb{C})^{\op}$ can be viewed as the Choi matrix of a quantum adjacency matrix, but also as an orthogonal projection onto a subspace of $M_n(\mathbb{C})$, when we identify the algebra of bounded operators on $M_n(\mathbb{C})$ (viewed as a Hilbert space) with $M_n(\mathbb{C}) \otimes M_n(\mathbb{C})^{\op}$ via the left-right action. 

For the trivial quantum graph, the corresponding subspace of $M_n(\mathbb{C})$ is the span of the identity matrix. It is therefore tempting to investigate other rank one subspaces. If such a subspace is spanned by a (suitably normalized) operator $T$, then the corresponding quantum adjacency matrix is equal to $A(x):= TxT^{\ast}$. This is the inspiration for the following definition.  Recall that the density matrices  $\{\rho_a:1\leq a\leq d\}$ associated to a $\delta$-form $\psi$ on $B$ are invertible and satisfy $\Tr(\rho_{a}^{-1})=\delta^2$.

\begin{definition}
 Let $T \in B$ satisfy $\Tr (\rho_a^{-1} T^{\ast}T) = \delta^2$ for all $1\leq a\leq d$. We call $A(x):= TxT^{\ast}$ a quantum adjacency matrix of \emph{rank one}.
\end{definition}

Before we proceed any further, we should first check that rank one quantum adjacency matrices are in fact quantum adjacency matrices. Because one can check the conditions separately on each matrix direct summand, we may assume that $B$ is a matrix algebra, so we can forget about the index $a$.
\begin{lemma}\label{lem:rankone}
If $T \in M_n(\mathbb{C})$ satisfies $\Tr (\rho^{-1} T^{\ast}T)=\delta^2$, where $\rho$ is the density matrix of a $\delta$-form $\psi$ then $A(x):=TxT^{\ast}$ is a quantum adjacency matrix.
\end{lemma}
\begin{proof}
We have to check that $m(A\otimes A)m^{\ast}(f_{ij}) = \delta^2 A(f_{ij})$ for all adapted matrix units $f_{ij}$. The left-hand side is equal to $\sum_{k} T f_{ik}T^{\ast}T f_{kj} T^{\ast}$. We will check that $\sum_{k} f_{ik} S f_{kj} = \Tr (\rho^{-1} S) f_{ij}$ and this will end the proof. First, we can write the matrix $S$ as $S = \sum_{pt} S_{pt} e_{pt} = \sum_{pt} S_{pt}\sqrt{\psi(e_{pp})\psi(e_{tt})} f_{pt}$. We have $f_{ik} f_{pt} f_{kj} = \delta_{pk}\delta_{tk} \frac{1}{\psi(e_{kk})^2} f_{ij}$, so $f_{ik} S f_{kj} = \frac{S_{kk}}{\psi(e_{kk})} f_{ij}$. We obtain
    \begin{align*}
        \sum_{k} f_{ik}S f_{kj} = \left(\sum_{k} \frac{S_{kk}}{\psi(e_{kk})}\right) f_{ij} = \Tr(\rho^{-1}S) f_{ij}. 
    \end{align*}
\end{proof}
\begin{proposition}\label{prop:triv}
Let $A$ be a rank one quantum adjacency matrix, given by $A(x)=TxT^{\ast}$. Recall that $T$ is of the form $\bigoplus_{a} T^{(a)}$ and we assume that each $T^{(a)}\neq 0$. Then the Cuntz--Pimsner algebra of the edge correspondence is isomorphic to $B\otimes C(\bT)$.
\end{proposition}
\begin{proof}
We will once again resort to \Cref{prop:recognition} to show that the edge correspondence is isomorphic to the trivial correspondence $B$. By our assumption on $T$, the element $\xi:=\frac{1}{\delta}T^{\ast}$ is cyclic for the trivial correspondence and we have $\delta^2 \langle \xi, x\xi\rangle = TxT^{\ast} = A(x)$. Therefore by \Cref{prop:recognition} the edge correspondence is isomorphic to the trivial correspondence by a bimodular extension of the assignment $\xi \mapsto \epsilon$. We conclude as in the proof of \Cref{cor:triv} that the corresponding Cuntz--Pimsner algebra is isomorphic to $B\otimes C(\bT)$.
\end{proof}
\begin{remark}
If it happens that $T^{(a)}=0$ for some $a$'s, then the same proof shows that the Cuntz--Pimsner algebra is isomorphic to $B' \otimes C(\bT)$, where $B'$ is the direct sum of matrix algebras, but only over the $a$'s, for which $T^{(a)}\neq 0$.
\end{remark}

\subsection{Quantum adjacency matrices arising from $*$-automorphisms} 
Another way to generalize the trivial quantum graph is to let $\alpha$ be a $*$-homomorphism of a finite dimensional C*-algebra $B$. As $\alpha$ is multiplicative, we have 
$
\alpha\circ m 
= m\circ (\alpha\otimes \alpha)
$. 
It follows that 
$
m(\alpha \otimes \alpha)m^{\ast} 
= 
\alpha mm^{\ast} 
= \delta^2 \alpha
$, i.e. $\alpha$ is a quantum adjacency matrix. Then $\cG_\alpha:=(B,\psi,\alpha)$ is a quantum graph, and the associated quantum edge indicator for $\cG_\alpha$ is 
\begin{align*}
    \epsilon_{\alpha}
:=\frac{1}{\delta^2}(\id\otimes \alpha)m^*(1)
    =\frac{1}{\delta^2}\sum_{aij}\psi(e_{ii}^{(a)})f_{ij}^{(a)}\otimes \alpha(f_{ji}^{(a)}).
\end{align*}

Let $E_{\alpha} = B \cdot \epsilon_\alpha \cdot B$ denote the quantum edge correspondence for $\cG_\alpha$, and define $B_{\alpha}$ to be $B$ as a C*-correspondence over itself with $\langle a,b\rangle=a^*b$, left action given by $x\cdot \eta:= \alpha(x) \eta$, and right action given by right multiplication.
\begin{lemma}
$E_\alpha\cong B_{\alpha}$ as C$^\ast$-correspondences.
\end{lemma}
\begin{proof}
The element $\xi:= \frac{1}{\delta} 1$ is cyclic for $B_{\alpha}$ and $\delta^2 \langle \xi, x\cdot \xi\rangle = \alpha(x)$. By \Cref{prop:recognition} the edge correspondence is isomorphic to $B_{\alpha}$.
\end{proof}
\begin{corollary}\label{cor:triv2}
When $\alpha$ is a $*$-automorphism, the Cuntz--Pimsner algebra of $E_\alpha$ is isomorphic to $B \rtimes_\alpha \mathbb{Z}$.
\end{corollary}
\begin{proof}
This follows from Pimsner's work \cite[Example 3, p. 193]{Pimsner}.
\end{proof}

We can make this crossed product a bit more explicit. First, recall that if we have an automorphism $\alpha$ on a direct sum $B_{1}\oplus B_2$, preserving the summands, then the corresponding crossed product splits as a direct sum as well, i.e. $(B_1\oplus B_2)\rtimes \mathbb{Z} \simeq (B_1\rtimes \mathbb{Z})\oplus (B_2\rtimes \mathbb{Z})$. Second,
automorphisms of direct sums of matrix algebras are of a very special form. We can first collect all the matrix algebras of the same size and such a subalgebra has to be preserved by any automorphism. In this case we are dealing with an algebra of the form $M_n(\mathbb{C}) \otimes \mathbb{C}^{k}$ and then any automorphism comes from a permutation of the set $\{1,\dots,k\}$ followed by \emph{inner} automorphisms on individual matrix algebras. Because crossed products are insensitive to inner perturbations (\cite[II.10.3.17]{MR2188261}), we may assume that we are just dealing with a permutation automorphism. Moreover any permutation decomposes as a disjoint sum of cycles so we can treat those separately. To sum up, any crossed product $B\rtimes \mathbb{Z}$ will decompose as a direct sum of crossed products of the form $(M_n(\mathbb{C})\otimes \mathbb{C}^k) \rtimes \mathbb{Z}$, where $\mathbb{Z}$ acts by a cycle on $\{1,\dots,k\}$ and as identity on $M_n(\mathbb{C})$. Since it acts as identity on the matrix algebra, the resulting crossed product is isomorphic to $M_n(\mathbb{C}) \otimes (\mathbb{C}^k \rtimes \mathbb{Z}) \cong M_n(\mathbb{C}) \otimes M_k(\mathbb{C}) \otimes C(\mathbb T)$ (\cite[Section VIII.3, p. 230]{MR1402012}. 

Again, there is a natural Exel system we may associate to the quantum graph $\cG_\alpha$. Consider $(B,\alpha,\alpha^{-1})$, where $\alpha^{-1}$ plays the role of the transfer operator. As in the example of the complete quantum graph, the Exel crossed product for this system, $B\rtimes_{\alpha,\alpha^{-1}} \mathbb{N}$, is isomorphic to the Cuntz--Pimsner algebra $B\rtimes_\alpha \mathbb{Z}$ for the quantum edge correspondence $E_\alpha.$





\bibliographystyle{amsalpha}
\bibliography{references}

\end{document}